\documentclass[sn-mathphys,Numbered]{sn-jnl}

\usepackage{graphicx}%
\usepackage{multirow}%
\usepackage{amsmath,amssymb,amsfonts}%
\usepackage{amsthm}%
\usepackage{mathrsfs}%
\usepackage[title]{appendix}%
\usepackage{xcolor}%
\usepackage{textcomp}%
\usepackage{manyfoot}%
\usepackage{booktabs}%
\usepackage{algorithm}%
\usepackage{algorithmicx}%
\usepackage{algpseudocode}%
\usepackage{listings}%

\usepackage{tikz}
\usepackage{pgfplots}
\usepgfplotslibrary{fillbetween}

\newtheorem{theorem}{Theorem}[section]
\newtheorem{proposition}[theorem]{Proposition}%
\newtheorem{lemma*}{Lemma}[section]%
\newtheorem{cor}{Corollary}[section]

\theoremstyle{remark}%
\newtheorem{remark}{Remark}[section]%

\theoremstyle{definition}%
\newtheorem{definition}{Definition}%

\newcommand{\N}{\Delta} 

\newcommand{\fg}{\mathfrak g}

\newcommand{\fn}{\mathfrak g}

\newcommand{\id}{{\rm id}}

\newcommand{\R}{\mathbb{R}}

\pgfplotsset{compat=1.17}

\begin{document}

\title[Optimal synthesis on path groups]{Sub-Riemannian optimal synthesis for Carnot groups with the structure of a path geometry}


\author{\fnm{Ale\v s} \sur{N\' avrat}, \fnm{Lenka} \sur{Zalabov\' a}}\email{ales.navrat@vutbr.cz}

\affil{\orgdiv{Institute of Mathematics}, \orgname{Faculty of Mechanical Engineering,  Brno University of Technology}, \orgaddress{\street{Technick\' a 2896/2}, \city{Brno}, \postcode{616 69}, \country{Czech Republic}}}


\abstract{This paper explicitly constructs the complete set of optimal sub-Riemannian geodesics starting from a point for certain Carnot groups of step two. These are groups of dimension 2n+1 equipped with a left-invariant distribution of dimension n+1 such that at each point, there is a unique direction defining a nontrivial Lie bracket. A suitable explicit expression of geodesics, together with symmetries of the structure, allows us to identify the cut time and the cut locus by applying the so-called extended Hadamard technique. 
}




\maketitle

\section{Introduction}
The research on optimal curves on Carnot groups is motivated by the study of nonholonomic robots and their optimal control, see \cite{hzn1,hzn2,hz}. Such mechanical systems are defined by nonholonomic constraints that induce a sub-Riemannian geometry on their configuration spaces, \cite{b96,J}. 
For solving the local control of a nonholonomic mechanism, its configuration space is approximated by a Carnot group, and a geodesic in this Carnot group describes the evolution of the mechanical system.

For an introductory understanding of sub-Riemannian geometry, see references  \cite{ABB,cg09}. Roughly speaking, this notion is a generalization of Riemannian geometry, where the main difference is that some directions of travel are explicitly forbidden. Similarly as Euclidean spaces are the most elementary examples of the Riemannian geometry, the simplest and the most symmetric examples of the sub-Riemannian geometry are Carnot groups. The key feature is the existence of a sub-Riemannian metric that measures the length of admissible curves, i.e., those curves whose tangents do not point in forbidden directions. The sub-Riemannian geodesics are such curves that locally minimize this length. Their properties are quite different from the Riemannian setting. In general, the problem of global optimality of geodesics is hard, even for Carnot groups. The solution is known only in a few cases, and it is based on specific symmetries of the given structure and fine analysis of functions defining the geodesics, see \cite{ABB,ABB12,mya1,mya2,Rizzi,Sac, talianky}. This paper aims to add a new class of Carnot groups to this series of solved examples. These are Carnot groups with the structure of a path geometry - a structure that may be viewed as a higher-dimensional generalization of the well-known Heisenberg group. Although it shares some common features with the Heisenberg case, new properties of geodesics emerge and lead to a different optimal synthesis.

Further in this section, a background on sub-Riemannian geometry and Carnot groups is given. 
The plan for the next sections is as follows. At first, the path geometry structure is described in detail, and its symmetries are found. Next, an explicit formula for sub-Riemannian geodesics and their characterization is obtained. In the last section, the results from the previous sections are applied to obtain the optimal synthesis via a generalization of the classical Hadamard technique from Riemannian geometry. The main result of this paper about the cut time and the cut locus is stated in Proposition \ref{cut_time_prop} and Theorem \ref{thm}.

\subsection{Sub-Riemannian geodesics on Carnot groups}
\label{SRgeodesics}
Let us recall that a sub-Riemannian manifold is defined by a  bracket-generating distribution in the tangent bundle determining the allowed directions of travel and by a metric on this distribution. A {\it Carnot group} $G$ is a connected and simply connected Lie group whose Lie algebra $\mathfrak{g}$ admits a decomposition
$$\mathfrak{g}=\mathfrak{g}_1\oplus \mathfrak{g}_2\oplus\cdots\oplus \mathfrak{g}_r,$$
satisfying $[\mathfrak{g}_1,\mathfrak{g}_i]=\mathfrak{g}_{i+1},$ where we set $\mathfrak{g}_{r+1}=0.$ The smallest such integer $r$ is called the {\it step} of the Carnot group. We will be interested in Carnot groups of step two, as mentioned above.
A choice of an inner product on the first layer $\mathfrak{g}_1$ automatically endows $G$ with a left-invariant sub-Riemannian structure $(G,\Delta,g)$, where $\Delta$ is bracket generating distribution defined by left-invariant vector fields whose value at the identity belongs to $\mathfrak{g}_1.$ The sub-Riemannian metric $g$ is defined on two vector fields in $\Delta$ by the inner product of their values at the identity.

Let us also recall that a Lipschitz curve $\gamma:\:[0,T]\to G$ is called horizontal if $\dot{\gamma}(t)\in\Delta_{\gamma(t)}$, and that its length is measured with respect to the sub-Riemannian metrics $g$. 
By {\it sub-Riemannian geodesics}, we mean a non-trivial horizontal curve parameterized by constant speed that locally minimizes the length between its endpoints. 
By viewing the geodesics as the solution to an optimal control problem, it is given by the projection of an extremal curve $\lambda\in T^*G$. Thanks to Goh condition, there do not exist so-called strictly abnormal curves in the case of two-step Carnot groups, see, e.g., \cite[Corollary 12.15]{ABB}. Therefore, each geodesic is the projection of a normal extremal parametrized by arclength, i.e., it is an integral curve of the Hamiltonian vector field corresponding to the sub-Riemannian Hamiltonian
$$
H: T^*G\to \mathbb{R},\quad H(\lambda)=\frac12 \sum_{i=1}^m\langle \lambda,X_i(q)\rangle^2,
$$
where  $q=\pi(\lambda)\in G$ is the projection of $\lambda\in T_q^*G$ to its base point, and where $\{X_1,\dots,X_m\}$ is a family of left-invariant vector fields generating distribution $\Delta$.
This Hamiltonian description  shows that geodesic curves issuing from a point $q_0$ are parameterized by the initial cylinder of normalized co-velocities
$$
\Lambda_{q_0}=T^*_{q_0}G\cap H^{-1}(1/2).
$$
Viewing these parameters as variables, we get the {\it exponential map}
$$
\exp_{q_0}:\mathbb{R}\times \Lambda_{q_0}\to G,\quad \exp_{q_0}(t,\lambda_0)=\gamma(t),
$$
where the geodesic curve $\gamma(t)$ is the projection of normal extremal $\lambda(t)$ such that $\lambda(0)=\lambda_0.$
Since $\exp_{q_0}(t,\lambda_0)=\exp_{q_0}(1,t\lambda_0)$ due to the homogeneity property of the Hamiltonian flow,  the exponential map can also be seen as a map $T_{q_0}^*G\to G$.

\subsection{Sub-Riemannian cut locus}
The geodesics on a sub-Riemannian manifold are locally length-minimizing by definition, but they usually lose their global optimality at a certain point.
The {\it cut time} of a geodesic is the critical time when the geodesic stops to be globally length-minimizing,
$$
t_{cut}(\gamma):=\operatorname{sup}\{t>0\;|\; \gamma_{[0,t]} \text{ is a minimizing geodesics}\}>0.
$$
The corresponding point $\gamma(t_{cut}(\gamma))$ on the horizontal curve is called the {\it cut point}. The {\it cut locus} $Cut_{q_0}$ of a point $q_0\in G$ is the set of all cut points of geodesics issuing from $q_0$. In the case of two-step Carnot groups, the cut point is either a Maxwell point or the first conjugate point; see Theorem 8.72 in \cite{ABB}. A {\it Maxwell point} is a point such that there exist (at least) two geodesics of the same length between $q_0$ and this point. Such points are closely related to fixed points of symmetries of the exponential map; see \cite{ABB,Rizzi,mya1}. On the other hand, the conjugate points correspond to critical points of the exponential map. Concretely, the {\it first conjugate point} to $q_0=\gamma(0)$ along $\gamma(t)=\exp_{q_0}(t\lambda_0)$  is the point $\gamma(t_{conj}),$ such that $t_{conj}$ is the {\it first conjugate time} defined by  $$t_{conj}(\gamma)=\operatorname{inf}\{t>0\;|\; t\lambda_0 \text{ is a critical point of } \exp_{q_0}\}.$$

The study of the cut and conjugate loci on surfaces is a classical topic in Riemannian geometry. Although the sub-Riemannian case has also been widely studied in the literature, they are fully described only in a few cases. A comprehensive summary can be found in \cite{Rizzi}. 
Here, we completely describe the cut locus of $(n+1,2n+1)$ Carnot groups equipped with the path geometry structure by applying the extended Hadamard technique, \cite[Section 13.4]{ABB}. This method is based on a classical result stating that if a smooth map between two connected manifolds of the same dimension is proper and its differential is nowhere singular, it is a covering.


\subsection{Heisenberg group} \label{Heisenberg}
We recall basic facts about the Heisenberg group since it is not only a basic example in sub-Riemannian geometry, but it will play an important role later in this paper. It is the lowest-dimensional nontrivial example of a Carnot group - the step is two, and the grow vector equals $(2,3)$.  The distribution is given by $\Delta=\operatorname{span}\{X_1,X_2\}$, and the vector fields $X_1,X_2,[X_1,X_2]$ span the whole tangent space in each point. 
Let us denote by $(x,\ell,y)$ the coordinates on $\mathbb{R}^3$. Then, the generating vector fields of the distribution are usually chosen such that 
\begin{align}
    \label{Heisenberg_fields}
    X_1=\partial_{x}+\frac{\ell}{2}\partial_y, \quad
    X_2=\partial_{\ell}-\frac{x}{2}\partial_y.
\end{align}
The distribution is bracket generating at every point since $[X_1,X_2]=-\partial_y$. The sub-Riemannian metric is defined by declaring $X_1,X_2$ to be orthonormal, i.e. we have $g=dx\otimes dx+d\ell\otimes d\ell$ in the chosen coordinates. The group law in the Heisenberg group reads
\begin{align}
\label{Heisenberg_multiplication}
\begin{pmatrix}
    x_1\\\ell_1\\y_1
\end{pmatrix}
\cdot
\begin{pmatrix}
    x_2\\\ell_2\\y_2
\end{pmatrix}
=\begin{pmatrix}
    x_1+x_2\\\ell_1+\ell_2\\y_1+y_2+\tfrac12(x_1\ell_2-\ell_1x_2)
\end{pmatrix}.
\end{align}
It is easy to see that the right multiplication sends horizontal curves to horizontal curves of the same length. In other words, it is a symmetry of the distribution, and it is also a symmetry of the exponential map. Consequently, each geodesic from a point $q_0$ is a translation of a geodesic issuing from the identity. In particular, the problem of optimal analysis is reduced to an analysis of geodesics issuing from the identity. The result is 
well known, see e.g., \cite{ABB, mope}. Namely, if the cylinder of initial normalized co-vectors is parameterized as $\Lambda_0=\{(\cos\alpha,\sin\alpha,\rho)\;|\; \rho\in \mathbb{R},\alpha\in S^1\},$ then for $\rho\neq 0$ the exponential map on Heisenberg group reads
\begin{align} \label{geodesics_Heisenberg}
\begin{split}
x(t)&=\frac{1}{\rho}\left(-\cos(\rho t+\alpha)+\cos\alpha\right),
\\
\ell(t)& =
\frac{1}{\rho}
\left(
\sin(\rho t+\alpha)-\sin\alpha 
\right),
\\
y(t)&=
\frac{1}{2\rho^2} \left(\rho t-\sin(\rho t)\right),
\end{split} 
\end{align}
and for $\rho=0$ one obtain simply straight lines $(\cos(\alpha) t,\sin(\alpha) t,0)$. A deeper analysis of formula \eqref{geodesics_Heisenberg} shows that the Heisenberg cut locus for geodesics issuing from the origin is the horizontal axis formed by points $(0,0,y).$ The cut time for a geodesic is when the geodesic meets this axis for the first time. In our parameterization, the cut time is given by
\begin{align} \label{cut_time_Heisenberg}
t_{cut} = \sqrt{4\pi y} = \frac{2\pi}{\rho}.
\end{align}
It is both the first conjugate time, $t_{conj}=t_{cut},$ and also the first Maxwell time. Indeed, the rotation around the horizontal axis is a  symmetry of the exponential map and leaves all points on this axis fixed. Thus, it generates an infinite number of geodesics of the same length to each cut point. Namely, any choice of parameter $\alpha\in S^1$ in \eqref{geodesics_Heisenberg} gives a different geodesics from the origin to cut point $(0,0,y),$ where $y=\pi/\rho.$

\section{The path geometry structure on a Carnot group}  \label{path_section}
Classically, a path geometry on a smooth manifold $M$ is given by a smooth family of unparameterized curves on $M$ such that for each point and each direction, there is a unique curve through this point in this direction. It can also be seen as an equivalence class of second-order ordinary differential equations 
\begin{align} \label{ODE}
y_i''(x)=f_i(x,y_j(x),y_j'(x)), \qquad i,j=1,\dots,n.
\end{align}
under point transformations. Indeed, regarding $(x,y_i)$ as local coordinates on $M$, the solutions to system  \eqref{ODE} locally give a family of immersed curves in $M$, with one curve in each direction. 
These curves (paths) lift to the projectivized tangent bundle $\mathcal{P}TM$, i.e., the space of all lines through the origin in $TM$. The lifts give rise to a line bundle $E$ in the tangent bundle of $\mathcal{P}TM$ with specific properties. Namely, the tautological subbundle of $T\mathcal{P}TM$ splits as $E\oplus V$, where $V$ is the vertical subbundle of $\mathcal{P}TM\to M$, and the Lie bracket of vector fields induces an isomorphism $E\otimes V \to T\mathcal{P}TM/(E\oplus V)$. The decomposition $E\oplus V$ can be used as an alternative definition of the path geometry and shows in turn that the structure is equivalent to the Cartan geometry of a parabolic type, see \cite{CS,Zadnik}.  

This paper studies sub-Riemannian structures on path geometries that are flat in the sense of Cartan, i.e., the most symmetric geometries.
In the above description, the flat model corresponds to the elementary second-order ODE system $y_i''(x)=0$; the associated paths are lines, and the point transformations are projective linear maps, hence one can set $M=\mathbb{R}P^{n+1}$. The lines in projective space canonically define curves in $\mathcal{P}T\mathbb{R}P^{n+1}$ that foliate this space, and thus their tangent spaces give rise to the line subbundle $E$.  Indeed, $\Delta=E\oplus V$ is a bracket generating distribution of rank $n+1$ on $\mathcal{P}T\mathbb{R}P^{n+1}$ that satisfies $$[V,E]=T\mathcal{P}T\mathbb{R}P^{n+1}/\Delta,$$
and all other Lie brackets vanish. 
In this paper, such a left-invariant distribution $\Delta=E\oplus V$ on a Carnot group is considered. Via group multiplication, the structure may be defined by certain properties of the corresponding Lie algebra.

\subsection{Definition and basic properties of the structure}
A Carnot group $G$ of step two with Lie algebra $\fg=\fg_{1}\oplus\fg_{2}$ has the structure of a path geometry if $\fg_{1}$ splits into the direct sum of a 1-dimensional component $\fg_{1}^E$ and an $n$-dimensional component $\fg_{1}^V$ such that the Lie bracket gives a linear isomorphism $\fg_{1}^E\otimes\fg_{1}^V\to\fg_{2}$. In particular it means that $\dim(\fg_{1}^V)=\dim(\fg_{2})=n$ hence the growth vector of such a Carnot group is $(n+1,2n+1)$.
We give a simple and explicit definition in terms of a basis of Lie algebra $\fg$. 
\begin{definition}
	A Carnot group $G$ of step two has the structure of a path geometry if it has a growth vector $(n+1,2n+1)$, where $n\in\mathbb{N}$, and if there exists a basis $\{e_0,e_1,\dots,e_{2n}\}$ of the associated Lie algebra $\fg$ such that the only nontrivial commutation relations are given by
	\begin{align} \label{com_relations}
	e_{n+i}=[e_i,e_0], \qquad i=1,\dots, n.
	\end{align}	
	In the sequel, such Carnot group $G$ will be called shortly a path group.
\end{definition}
In other words, the path group $G$ is a Carnot group such that the structure constants of the corresponding Lie algebra $\fn$ satisfy $c_{i0}^j=-c_{0i}^j=\delta_i^j$ for each $i,j=1,\dots,n$  and they are zero otherwise.
%
%
 In our basis, the Lie algebra decomposition $\fg=\fg_{1}^E\oplus\fg_{1}^V\oplus\fg_{2}$ described above is given by
\begin{align} \label{path_basis}
\fg_{1}^E=\operatorname{span}\{e_0\},\; \fg_{1}^V=\operatorname{span}\{e_1,\dots,e_n\},\;	\fg_{2}=\operatorname{span}\{e_{n+1},\dots,e_{2n}\}.
\end{align} 
The path group $G$ is endowed with the left-invariant sub-Riemannian structure $(G,\N,g)$, where $\N$ is a $(n+1)$-dimensional distribution equipped with a natural decomposition
\begin{align}
\label{decom}
\N=E\oplus V = \operatorname{span}\{X_0\}  \oplus \operatorname{span}\{X_1,\dots,X_{n}\},
\end{align}
where $X_0,X_1,\dots,X_{n}$  are left-invariant vector fields corresponding to $\fg_1^E\oplus\fg_1^V$. These vector fields, together with vector fields defined by
\begin{align}
Y_{i}=[X_i,X_0], \qquad i=1,\dots,n
\end{align}
form a privileged frame for the distribution $\Delta$.
Using the canonical coordinates of the first kind defined by the exponential map from $\fg$ onto $G$, the path group $G$ can be identified with $\R^{n+1}\oplus\R^n$, where the first layer further splits as $\R^{n+1}=\R\oplus\R^n$. Namely, the coordinates map is given by
\begin{align} \label{coordinates}
    \exp(xX_0+\ell_iX_i+y_iY_i) \mapsto (x,\ell_i,y_i),
\end{align}
where $i=1,\dots,n$. Hence, each point in $G$ can be written as a vector $(x,\ell,y)$ such that $(x,\ell)\in\R\oplus\R^n$ is defined by the first stratum of $\fg$ while $y\in\R^n$ corresponds to the second stratum. The group law on $G$ can be easily computed by Baker-Campbell-Hausdorff formula.  In vector notation, it reads 
\begin{align}\label{grupa}
\begin{pmatrix} 
x \;\;|\; \ell \\ y
\end{pmatrix}
\cdot 
\begin{pmatrix} 
\tilde x \;\;|\;\; \tilde \ell \\ \tilde y
\end{pmatrix}
=
\begin{pmatrix} 
x+  \tilde x \;\;|\;\; \ell + \tilde \ell\\
y+ \tilde y + \frac{1}{2} (\ell \tilde x -x \tilde \ell)
\end{pmatrix}.
\end{align}
Note that we chose this notation to emphasize that both $x\in\R$ and $\ell\in\R^n$ lie in the first layer and that $G$ is of step two.
Differentiating this formula, we get the expression for the left-invariant privileged frame in the canonical coordinates
\begin{align} \label{fields_basis}
\begin{split}
X_0&=\partial_x+\sum\limits_{i=1}^{n}\frac{\ell_i}{2}{\partial_{y_i}}, \quad X_i=\partial_{\ell_i}-\frac{x}{2} \partial_{y_i}, \quad i=1,\dots,n, \\
Y_{i}&= \partial_{y_i},\quad i=1,\dots,n,
\end{split}
\end{align}
where the symbol $\partial$ stands for partial derivative.  Indeed, it is easy to check that these vector fields satisfy commutation relations \eqref{com_relations}. The sub-Riemannian metric $g$ is given by declaring $X_0,\dots,X_n$ orthonormal. In our coordinates, we have
\begin{align}
	\label{SR_metric}
g=dx \otimes dx +\sum\limits_{i=1}^nd\ell_i \otimes d\ell_i.
\end{align}
Indeed,  $g(X_a,X_b)=\delta_{ab}$ for each $a,b=0,\dots,n.$ The path geometry structure can also be described by horizontal condition $(\dot x,\dot \ell, \dot y)\in\Delta$ for a curve in $G$. Using the basis  of $\Delta$ given by \eqref{fields_basis} we see that  
horizontal curves are such curves that satisfy
\begin{align}
\label{horizontal_curve}
\dot y = \frac12 (\dot x \ell - \dot \ell x).
\end{align}


\begin{remark}
	 Note that the privileged coordinates are chosen such that the left-invariant vector fields \eqref{fields_basis}  directly generalize the usual basis of the Heisenberg group. However, there is also a natural choice of the basis motivated by the identification of $G$ with $\mathcal{P}T\mathbb{R}P^{n+1}$ described above. Namely, choosing the canonical coordinates  $(x,y_i,\ell_i)$ on $\mathcal{P}T\mathbb{R}P^{n+1}$, we get
\begin{align*}
E
=\operatorname{span}\{\partial_x+\sum_{i=1}^{n}\ell_i\partial_{y_i}\}, \quad V=\operatorname{span}\{\partial_{\ell_i}\},\quad T\mathcal{P}T\mathbb{R}P^{n+1}/(E\oplus V)=\operatorname{span}\{\partial_{y_i}\}.
\end{align*} 
%
\end{remark}

\subsection{Symmetries of the structure} 
By symmetries of a sub-Riemannian structure $(G,\N, g)$, we mean automorphisms of $G$ preserving the distribution $\N$  and sub-Riemannian metric $g$. The continuous symmetries in the connected component of identity can be described as flows of infinitesimal symmetries, i.e., the flows of vector fields $v$ such that $\mathcal{L}_v(\Delta) \subset \Delta$ and $\mathcal{L}_vg=0$. The advantage of passing to the Lie algebra level is that such infinitesimal symmetries can be algorithmically constructed by so-called Tanaka prolongation (or by another prolongation technique), see \cite{tan}. The algebraic prolongation shows that the infinitesimal symmetries are isomorphic to $\fn\oplus\mathfrak{so}(n)$. The first summand corresponds to the right-invariant vector fields, while the infinitesimal symmetries isomorphic to $\mathfrak{so}(n)$ are represented by vector fields  
	\begin{align}\label{so3}
	r_{ij}=\ell_i\partial_{\ell_j}-\ell_j\partial_{\ell_i}+y_i\partial_{y_j}-y_j\partial_{y_i},
	\quad i<j \text{ and } i,j=1,\dots,n.
	\end{align} 
Indeed, it is easy to check that it preserves both the distribution and the metric. 
The vector fields $r_{ij}$ integrate to symmetries that act as simultaneous rotations in planes $\langle\ell_i,\ell_j\rangle$ and $\langle y_i,y_j\rangle$ while they leave the coordinate $x$  invariant. 

\begin{proposition}
There are two families of continuous symmetries on a path group $G=(\R\oplus\R^n)\oplus\R^n$. (a) Non stabilizing: the right translation by $q^{-1}$ given by \eqref{grupa} is a symmetry that maps $q$ to the origin for each $q\in G$.
	(b) Stabilizing the origin: For each $R \in SO(n)$, there is a symmetry that, in the vector notation, reads
	\begin{align}
	\label{action}
	\begin{pmatrix} 
	x \;\;|\;\; \ell \\ y
	\end{pmatrix} 
	\mapsto 
	\begin{pmatrix} 
	x \;\;|\;\; R\ell \\ Ry
	\end{pmatrix}.
	\end{align}
	\label{corsym}
\end{proposition}
\begin{proof}
  (a) The right translations are obvious symmetries of each left-invariant structure. (b) The simultaneous rotations preserve the distribution since they evidently leave the condition for horizontal curves \eqref{horizontal_curve} invariant. By definition of the orthogonal group, they also preserve the sub-Riemannian metric \eqref{SR_metric}.
\end{proof}
The translations will allow us to discuss only the properties of sub-Riemannian geodesics issuing from the identity element
 since they act transitively and freely on $G$. On the other hand, the rotational symmetries will help us to reduce the dimensions of spaces when solving optimality questions.
Note that we do not discuss the discrete symmetries now. Actually, we will use one of the discrete symmetries of the exponential map later to find the cut locus and to resolve the problem of global optimality of geodesics on $G$.

\subsection{The action of $SO(n)$ on $G$} \label{orbits}
Let us discuss the stabilizers of points in $G$ under the action of symmetries from Proposition \ref{corsym}.  The right translations obviously do not stabilize any point of $G$. 
Looking at formula \eqref{action} for the action of $SO(n)$ we realize that there are three types of points with different (non-isomorphic) stabilizers. Namely, for a point  $q=(x,\ell,y)\in G$ it depends on whether vectors $\ell,y\in\R^n$ are zero or nonzero linearly dependent or linearly independent - the points with dependant vectors have a larger stabilizer. A concrete description of the stabilizers, orbits, and quotient spaces is summarized in the following proposition.
\begin{proposition} \label{factor_spaces}
	Under the action of $SO(n)$, the path group $G=\R\oplus\R^n\oplus\R^n$ decomposes into three invariant sets $G=G_0\sqcup G_1 \sqcup G_2$, where 
	\begin{align} \label{G1}
	G_k=\{(x,\ell,y)\in G \;|\; \dim(\operatorname{span}\{\ell,y\})=k\}, \; k=0,1,2.
	\end{align}
	The group $SO(n)$ acts trivially on $G_0=\id$ while for $G_1$ and $G_2$  we have a description as follows.
	\begin{enumerate}
		\item The stabilizer of each point in $G_1$ is isomorphic to $SO(n-1)$, and the quotient space $G_1/SO(n)$ is parameterized by three invariants $(x,|\ell|,|y|)$, where $|\phantom{.}|$ denotes the Euclidean norm in $\R^n$.
		\item The stabilizer of each point in $G_2$ is isomorphic to $SO(n-2)$, and the quotient space $G_2/SO(n)$ is parameterized by four invariants $(x,|\ell|,|y|, \varphi)$, where $\varphi\in (0,\pi)$ is the angle between vectors $\ell,y$. Alternatively, one can use invariants $(x,\ell^2,\ell\cdot y,\ell\wedge y)$ or $(x,\ell^2,\ell\cdot y,y^2)$, where $\cdot$ stands for the standard Euclidean scalar product in $\mathbb R^n$.
	\end{enumerate}
\end{proposition}
\begin{proof}
(1) The  algebraic set $G_1\subset G$ is obviously $SO(n)$-invariant and, as a submanifold, it has dimension $n+2$ since it is defined by $x,\ell$ and the ratio $y/\ell$. By \eqref{action}, the stabilizer of a point $q=(x,\ell,y)\in G_1$ equals the stabilizer of vector $\ell\in\mathbb{R}^n$ under the standard action of $SO(n)$. It consists of rotations in the subspace orthogonal to vector $\ell$, hence is isomorphic to $ SO(n-1)$. Then the orbits in $G_1$ are homogeneous spaces $SO(n)/SO(n-1)\cong S^n$. In particular, they have dimension $n-1$, and thus the corresponding quotient space $G_1/SO(n)$ is parameterized by $n+2-(n-1)=3$ scalar invariants.  A natural choice reads $(x,|\ell|,|y|)$ since $x$ is invariant and the standard action of the orthogonal group on $\R^n$ preserves the Euclidean length of vectors.

(2) The $SO(n)$-invariant set $G_2$ is defined by $\ell\wedge y\neq 0$ and thus is open in $G$ and its dimension is equal to $\dim(G)=2n+1.$ The stabilizer of a point  $q=(x,\ell,y)\in G_2$ under the action of $SO(n)$ is formed by rotations of the orthogonal complement of two-dimensional subspace generated by vectors $\ell,y\in\mathbb{R}^n,$ hence is isomorphic to $SO(n-2).$ The orbits in $G_2$ are then isomorphic to $SO(n)/SO(n-2)$.
The dimension of this orbit is $2n-3$, thus the quotient space $G_2/SO(n)$ is parameterized by $2n+1-(2n-3)=4$ invariants. By this description, a natural choice for independent invariants is $x$, lengths of vectors $\ell,y$, and the angle $\varphi$ between them. Alternatively, one can use $(x,\ell^2,\ell\cdot y,\ell\wedge y)$ or $(x,\ell^2,\ell\cdot y,y^2)$. The transitions between these sets of coordinates are given by well-known relations  $\ell\wedge y=|\ell||y|\sin\varphi$, $\ell\cdot y=|\ell||y|\cos\varphi$ and $(\ell\wedge y)^2=y^2\ell^2-(\ell\cdot y)^2$.
\end{proof}

\section{Sub-Riemannian geodesics of path geometries} \label{geodesics_section}
We use the standard method based on Pontryagin maximum principle, \cite{ABB}. Since path geometry $(G,\Delta,g)$ is of step two, geodesics are projections of normal Pontryagin extremals, i.e., integral curves of Hamiltonian vector field corresponding to the normal left-invariant sub-Riemannian Hamiltonian 
\begin{align} \label{ham}
H=\frac12(h_0^2+h_1^2\cdots+h_{n}^2),
\end{align}
where the functions $h_0,h_i$ are defined  by evaluating $\lambda\in T^*G$ on a basis of $\Delta$.  Namely, we consider $h_0 (\lambda )=  \lambda(X_0)$, $h_i (\lambda )=  \lambda(X_i),$ where $i=1,\dots,n$, and  $X_0,X_i\in TG$ are the vector fields given by \eqref{fields_basis}. Moreover, the arc-length parameterized geodesics satisfy $H=1/2$ along the solution. As $G$ is a Carnot group, the corresponding Hamiltonian system can be solved by a simple integration. This leads to an explicit formula for geodesics, which we will use in the sequel for their optimality analysis.

\subsection{An explicit formula for geodesics} 
We follow here \cite[Sections 7 and 13]{ABB}. 
Due to its left-invariance, the normal Hamiltonian \eqref{ham} can be considered as a function on $\fn^*.$ Under identification $(\fn^*)^*=\fn$ we have $dh_j=e_j$ for all $j=0,\dots,n$ and $dH=\sum_{j=0}^n h_j e_j\in \fn.$  
Then the Hamiltonian equations on $T^*G$ can be written in a compact and coordinate-free way for $g\in G$ and $\xi=\sum_{j=0}^n h_j e_j^*\in\fn^*$ as
\begin{align}
\begin{split} \label{ham_eqs}
\dot{g}&=L_{g*}dH,\\
\dot{\xi}&=(ad\;dH)^*\xi,
\end{split}
\end{align}
where  $L_{g*}$ is the tangent map to left multiplication and $ad^*$ is the coadjoint action of $\fn$ on $\fn^*.$
To get an explicit formula for its solution we rewrite this system in privileged coordinates on $G$, i.e., we consider  $g=(x,\ell_i,y_i)$, $i=1,\dots,n$, as in \eqref{fields_basis}. By formula \eqref{grupa} for the group multiplication, the horizontal system in the vector notation reads  
\begin{align}
\begin{split} \label{xl}
\begin{pmatrix}
\dot x \\
\dot \ell
\end{pmatrix}
&=
\begin{pmatrix}
  h_0 \\
  h
\end{pmatrix},\\
\dot y &= \frac12(h_0\ell -xh).
\end{split}
\end{align}
For the vertical system we consider coordinates $\xi=(h_0,h_i,w_i)$ on $\fn^*$, and we use the structure of Lie algebra $\fn$ given by formula \eqref{com_relations}. Then, in the vector notation, we get
\begin{align} 
\begin{split} \label{ver_obec} 
\begin{pmatrix}
\dot h_0 \\ \dot h
\end{pmatrix}&= \begin{pmatrix}0 & w^t \\ -w & 0\end{pmatrix} \begin{pmatrix}
h_0 \\  h
\end{pmatrix}, \\
\dot w&=0. 
\end{split}
\end{align}

The solution of this Hamiltonian system can be easily found by integrating the equations in a reversed order.  
Indeed, the second equation in \eqref{ver_obec} says that the vector $w$ is constant with respect to $t$. Thus, the first part of the vertical system is homogeneous with constant coefficients and its solution is given by the exponential of the defining matrix. 
\begin{lemma*} \label{vert-p} The solution of fiber system \eqref{ver_obec} is either constant or it is given by 
\begin{align} 
\begin{split}\label{h1}
h_0(t)&=\frac{|K|}{\rho}\sin(\alpha+|K|t), \\ 
h(t)&=
\frac{K}{\rho}\cos(\alpha+|K|t) + K^\perp, \\
w(t)&=K 
\end{split} 
\end{align}
where $K\in\R^n$ is nonzero, $|K|$ is its euclidean norm, $\alpha\in S^1$, $\rho\in\R^+$ and where $K^\perp\in\R^n$  is a vector orthogonal to $K$ with respect to the standard inner product.
\end{lemma*}
\begin{proof}
If $w(0)=0$ holds then $w(t)=0$ and $h_0(t)$, $h(t)$ are constant. If $w(0)\neq 0$ then $w(t)=K$ for a nonzero constant vector $K\in\R^n$ and  $h_0(t)$, $h(t)$ are given by the exponential of matrix
\begin{align}
\Omega=
\begin{pmatrix} 
0 & K^t\\
-K & 0 \\
\end{pmatrix}.
\label{matrix}
\end{align}
Computing the characteristic polynomial of this matrix,
we find the eigenvalue $0$ of multiplicity $n-1$ and imaginary eigenvalues $\pm i|K|$, where $|K|$ denotes the norm of vector $K$,
\begin{align*}
\text{spec}(\Omega)=\{\pm i |K|,0\}.
\end{align*} 
The eigenspace to eigenvalue $\pm i|K|$, viewed as a subspace of $\R^{n+1}=\R\oplus\R^n$, is generated by  eigenvector $\left( \pm i |K|,K\right)$ while the kernel of $\Omega$ is formed by vectors $(0,K^\perp)$, where $K^\perp\in\R^n$ is  perpendicular to vector $K$. Considering the real basis of the two-dimensional complement to the kernel, the solution to the first part of \eqref{ver_obec} can be written as
\begin{align*}
\begin{pmatrix} 
h_0(t)\\
h(t)
\end{pmatrix} 
&=
(C_1\cos(|K|t)+C_2\sin(|K|t))
 \begin{pmatrix} 
|K| \\
0
\end{pmatrix} 
\\
&+
(-C_1\sin(|K|t)+C_2\cos(|K|t)) \begin{pmatrix} 
0 \\
K
\end{pmatrix} 
+\begin{pmatrix} 
0 \\
K^\perp
\end{pmatrix},
\end{align*}
where $C_1,C_2$ are arbitrary real constants. We may assume that at least one of these constants is nonzero since we get a constant solution otherwise. The equation \eqref{h1} is obtained by taking polar coordinates  $C_1=(1/\rho)\sin\alpha$, $C_2=(1/\rho)\cos\alpha$ in the plane of these parameters and by using standard trigonometric formulas.

\end{proof}
\begin{remark}
	The relation between coefficients $\rho,\alpha$ from the  proposition above to the initial values of $h_0$ and $h$ in the origin is given by
	 $(1/\rho)\sin\alpha=h_0(0)/|K|$ and $(1/\rho)\cos\alpha=1/|K|^2(h(0)\cdot K)$ while $K^\perp$ is the projection of initial value $h(0)$ to the orthogonal complement of $K$, namely 
	$$K^\perp=h(0)-\frac{h(0)\cdot K}{|K|^2} K.$$
\end{remark}

Now, an explicit description of sub-Riemannian geodesics is obtained by substituting the solution of fiber system from lemma \ref{vert-p} into the horizontal system, followed by the integration.  Note that, as discussed above, we are interested in solutions issuing from the origin.
\begin{proposition} \label{geodesics}
	Let $G$ be the path group with a left-invariant distribution defined by vector fields \eqref{fields_basis} in privileged coordinates $(x,\ell_i,y_i),$ $i=1,\dots,n$. In these coordinates, the sub--Riemannian geodesics on $G$ issuing from the identity are  

	\noindent (1)
lines of the form 
 \begin{align} 
 \begin{split}
 x(t)&=C_0 t \\
 \ell(t)&= C t\\
 y(t)&=0
 \end{split}
 \label{line} 
 \end{align}
  parameterized by $C_0\in\R$, $C\in\R^n$ such that 
$
C_0^2+|C|^2=1,
$

  \noindent (2)
  and curves that are given by equations 
\begin{align} \label{horizontal}
\begin{split}
x(t)&=\frac{1}{\rho}\big(\cos\alpha-\cos(\alpha+|K|t)\big),
\\
\ell(t)& =
\frac{K}{|K|\rho}
\big(\sin(\alpha+|K|t)-\sin\alpha\big)
+{K^\perp}t, 
\\
y(t)&=\frac{K}{2|K|\rho^2}\big(|K|t-\sin(|K|t)\big)+
\frac{K^\perp}{2|K|\rho} \big( 2(\sin(\alpha+|K|t)-\sin\alpha) \\
& -|K|t(\cos(\alpha+|K|t)+\cos\alpha)  \big),
\end{split} 
\end{align}
 parameterized by mutually orthogonal vectors $K,K^\perp\in\R^n$, $K\neq 0$, and constants $\rho\in\R^+, \alpha\in S^1$  such that
\begin{align} \label{level}
\frac{|K|^2}{\rho^2} +|K^\perp|^2 =1.
\end{align}

\label{geo}
\end{proposition}
\begin{proof}
The system of lines \eqref{line} corresponds to constant solutions of the horizontal system; $C_0=h_0(0), C=h(0)$.
The equation \eqref{horizontal} is obtained by substituting the solution \eqref{h1} of the vertical system into the horizontal system  \eqref{xl} and by integration.
Recall that geodesics are contained in the level set $H=1/2$, i.e., $h_0^2+|h|^2=1$. According to \eqref{h1} and the orthogonality of $K$ and $K^\perp$  this restriction reads as \eqref{level}.
\end{proof}
 
We are mainly interested in sub-Riemannian geodesics emanating from the origin that are not lines, i.e., the geodesics given by  \eqref{horizontal} in Proposition \ref{geodesics}. An important observation following from these equations is that both the $\ell$-part and the $y$-part of such a geodesic lie in a common two--dimensional subspace in $\R^n$ generated by orthogonal vectors $K,K^\perp$. To get simpler equations for the geodesics, we express curves $\ell(t)$ and $y(t)$ in an orthonormal basis formed by unite vectors $k=K/|K|$ and $k^\perp=K^\perp/|K^\perp|$ in this order. 
\begin{cor}
	\label{geodesics_cor}
	 Sub--Riemannian geodesics on path group $G$ issuing from the origin are either lines parametrized by the unite speed in subspace $\{y=0\}$ or curves  given by 
		\begin{align} \label{geodesics_polar}
	\begin{split}
	x(t)&=\frac{1}{\rho}\left(\cos\alpha-\cos(\tau+\alpha)\right),
	\\
	\ell(t)& =
	\frac{1}{\rho}
	\begin{pmatrix}
	\sin(\tau+\alpha)-\sin\alpha \\
	\sigma\tau
	\end{pmatrix},
	\\
	y(t)&=
	\frac{1}{\rho^2}
	\begin{pmatrix}
	\frac12\left(\tau-\sin\tau\right)\\
	\sigma\left(\sin(\tau+\alpha)-\sin\alpha-\frac{\tau}{2}\left(\cos(\tau+\alpha)+\cos\alpha\right)\right) 
	\end{pmatrix},
	\end{split} 
	\end{align}
	  where $\tau=\rho t/\sqrt{1+\sigma^2}$ and where by $\ell(t),y(t)\in\R^2$ we mean coordinates of the corresponding curves in $\R^n$ expressed in an orthonormal basis $(k,k^\perp)$. The curves are parameterized by 
	 $\alpha\in S^1$ and $\rho,\sigma\in\R$ such that $\rho>0,\sigma\geq 0.$ 
\end{cor}
\begin{proof}
Let us  introduce a new parameter  $\sigma=\rho|K^\perp|/|K|\geq 0$ and a rescaled 'time'  $\tau=|K|t$. The equation \eqref{geodesics_polar} is obtained by substituting $|K^\perp|$ and $t$ in \eqref{horizontal} by these new parameters and by expressing curves $\ell(t),y(t)\in\R^n$ in the orthonormal basis formed by unite vectors $K/|K|$ and $K^\perp/|K^\perp|$.
The constrain  \eqref{level} reads $|K|^2=\rho^2/(1+\sigma^2)$ using the new parameters, hence the rescaled time equals $\tau=|K|t=\rho t/\sqrt{1+\sigma^2}$.
\end{proof}

\subsection{A characterization of geodesics}
It is easy to see from expression \eqref{geodesics_polar} that the projection of geodesic curves to the first layer $(x,\ell)$ is a helix.  The parameter $\alpha$ determines a shift of this helix from the origin, $\sigma$ gives its pitch, and $\rho$ is the dilatation. The remaining parameters of geodesics are the orthonormal vectors $k,k^\perp$ that span two-dimensional subspace of $\R^n$ in which the vectors $\ell(t), y(t)$ lie. More precisely, the dimension of this space is two for $\sigma\neq 0$ (that corresponds to $K^\perp\neq 0$ by definition) and one for $\sigma=0$. In such case,
 the whole geodesic lies in a three-dimensional space $\R\oplus \text{span}\{k\}\oplus \text{span}\{k\}$, and its equation \eqref{geodesics_polar} coincide with the well-known formula \eqref{geodesics_Heisenberg} for a geodesic on the Heisenberg group. Obviously, all points lying on such a geodesic satisfy $l(t)\wedge y(t)=0$ and thus the whole geodesic curve lies in the $SO(n)$-invariant set $G_1$, see \eqref{G1}. On the other hand, for $\sigma\neq 0$ we are going to show that the whole geodesic lies in $G_2,$ i.e., $l(t)\wedge y(t)\neq 0$ for all $t>0$.

\begin{lemma*} \label{lema}
	Let $q(t)=(x(t),\ell(t),y(t))$ be a geodesic on $G$ starting at the origin. Then  $\ell(t) \wedge y(t)\neq 0$ for all $t>0$ unless $\sigma=0$.
\end{lemma*}
\begin{proof}
	We show that the determinant formed by vectors $\ell(t),y(t)\in\R^2$ is strictly negative for all $\tau>0.$
	First, we make a next simplification of the geodesic equations \eqref{geodesics_polar} by using trigonometric formulas that convert sums to products. Namely, using a new 'time' $\tau$ for $\tau/2$, we get the following equations for geodesics 
	\begin{align} \label{geodesics_greek}
\begin{split}
x(t)&=\frac{2}{\rho}  \sin \tau  \sin \! \left(\tau +\alpha \right),
\\
\ell(t)& =\frac{2}{\rho}\begin{pmatrix}
\sin \tau  \cos \! \left(\tau +\alpha \right) 
\\
\tau  \sigma  
\end{pmatrix},
\\
y(t)&=
\frac{1}{\rho^2}\begin{pmatrix}
\tau -\cos\tau  \sin\tau 
\\
2\sigma  \cos\! \left(\tau+\alpha \right) \left(\sin\tau-\tau  \cos \tau\right) 
\end{pmatrix},
\end{split} 
\end{align}
	where $\tau=\rho t/2\sqrt{1+\sigma^2}.$ Now we compute
	$ 
	\operatorname{det}(\ell,y)=-\frac{2\sigma}{\rho^3} f(\tau,\alpha),
	$
	where we put
	\begin{align} \label{f_det}
	f(\tau,\alpha)=\tau^2-\tau\sin\tau\cos\tau-2\sin\tau\left(\sin\tau-\tau\cos\tau\right)\cos^2(\tau+\alpha).
	\end{align}
	Since $\rho,\sigma>0$, it is sufficient to prove that the function $f$ of two variables given by \eqref{f_det} is strictly positive for all $\tau>0$.
	This assertion can be proved by using inequality  $0\leq\cos^2(\tau+\alpha)\leq 1$. Indeed, substituting the lower and the upper bound in $f$ we see that its values  lie  between values of functions of one variable $f_0$ and $f_1$, where 
	\begin{align*}
	f_0(\tau)&=\tau(\tau-\sin\tau\cos\tau), \\
	f_1(\tau)&=\tau^2+\tau\sin\tau\cos\tau-2\sin^2\tau,
	\end{align*} 
	and both functions vanish for $\tau=0$ and are strictly positive for all $\tau>0.$
	The positivity of $f_0$ follows from the obvious inequality $\tau>\sin\tau\cos\tau$. The positivity of $f_1$ can be proved, for example,  by a combination of a 'local' estimate of goniometric functions by their truncated Taylor series in $\tau=0$ and their 'global' estimate  $-1<\sin\tau,\cos\tau<1$, as in \cite{hz}, or by a similar technique as in \cite{talianky}. 
\end{proof}

This lemma, together with the facts stated above, allows us to characterize the sub-Riemannian geodesics on $G$ as follows.
Recall that the group splits into disjoint sets $G=\id\sqcup G_1\sqcup G_2$ as described in Proposition \ref{factor_spaces}. 
\begin{proposition} \label{characterization}
	 The sub--Riemannian geodesics on path group $G$ issuing from the identity may be of the following three types.
	 \begin{enumerate}
	 	\item  Lines parameterized by unite speed in subspace $\{y=0\}$.
	 	\item  Heisenberg geodesic lying in a three-dimensional subspace of $G_1$. They are parameterized by $\rho>0,\alpha\in S^1$ and by a unite vector $k\in\R^n$ that determines the three-dimensional subspace $\R\oplus \text{span}\{k\}\oplus \text{span}\{k\}$.
	 	\item Curves in a five-dimensional subspace of $G_2$ given by \eqref{geodesics_polar} or equivalently by \eqref{geodesics_greek}. They are parameterized by $\rho,\sigma>0$, $\alpha\in S^1,$  and by two unite orthogonal vectors $k,k^\perp\in\R^n$ that define the five-dimensional subspace $\R\oplus \text{span}\{k,k^\perp\}\oplus \text{span}\{k,k^\perp\}$.
	 \end{enumerate}
\end{proposition}
\begin{remark}
	The equivalence between geodesics in $G_1$ and the Heisenberg geodesics can be seen immediately from the description of the path geometry structure in section 1. Indeed, the push forward of left-invariant frame \eqref{fields_basis} on $G_1$ along the factorization map $G_1\to \R^3$, defined by $(x,\ell e_1,y e_1)\mapsto (x,\ell,y)$, reads
	\begin{align*}
	\begin{split}
	X_0^*=\partial_{x}-\frac{\ell}{2}{\partial_{ y}},\quad
	X_i^*=\partial_{ \ell}+\frac{x}{2} \partial_{ y},\quad
	Y_i^*=\partial_{y}
	\label{vfH}
	\end{split}
	\end{align*}
	which is the standard form of generators of Heisenberg Lie algebra. 
	Hence, the group law \eqref{grupa} on $G_1$ gives an isomorphism  $G_1/SO(n) \cong \mathbb H_3$.
\end{remark}
A similar description holds for geodesics between arbitrary two points $q_0,q_1\in G.$ Concretely, let us consider a geodesic between points $q_0=(x_0,\ell_0,y_0)$ and $q_1=(x_1,\ell_1,y_1)$. Using the right translation by $q_0^{-1}$, we get a geodesic from the origin to the point 
	\begin{align*}
	\tilde{q}=q_1q_0^{-1}=\begin{pmatrix}
	x_1-x_0 \\ \ell_1-\ell_0 \\ y_1-y_0 +\frac12(x_1\ell_0-\ell_1 x_0)
	\end{pmatrix},
	\end{align*}
	where we used the formula \eqref{grupa} for group multiplication. For this geodesic, the characterization given in Proposition \ref{characterization} holds. Hence, each geodesic between points $q_0,q_1$ is one of the three types listed in this proposition translated by right multiplication by $q_0$.

\section{Optimality of geodesics} \label{sec4}

Our optimality analysis of sub-Riemannian geodesics on Carnot group $G$ with the path geometry structure is based on explicit formulas for geodesics obtained in section \ref{geodesics_section} and symmetries of the structure described in section \ref{path_section}. In particular, the right invariant vector fields allow us to restrict the analysis to geodesics starting in the origin which we characterized above. From the three types given in Proposition \ref{characterization} the lines are obviously always optimal while the geodesics in $G_1$ are equivalent to Heisenberg geodesics, and their cut locus and cut time are well known. Hence, it remains to study the optimality of geodesics in $G_2$. 

Note that this case is quite different from the Heisenberg case or the case of the $(3,6)$ free Carnot group since the $SO(n)$ symmetry does not produce any Maxwell point. Indeed, suppose that an end-point $q$ of a geodesics $q(t)=(x(t),\ell(t),y(t))$ in time $t=T$ is a fixed point for such a symmetry. Then, the symmetry acts trivially on the subspace generated by vectors $\ell(T),y(T)\in\R^n$ (whose dimension is either one or two). 
According to the above characterization of geodesics, we know that vectors $\ell(t),y(t)$ lie in this subspace for all times, and thus the whole geodesics lies in the fixed points of the given symmetry. In other words, the symmetry does not produce different geodesics with end-point $q$.
  On the other hand, it means that we may use these symmetries to reduce the problem of the optimality of geodesics to the problem of optimality of their projections to factor space $G_2/SO(n)$.

\subsection{Factorization of the exponential map}
The equations for sub-Riemannian geodesics define the exponential map
$\exp: \mathbb{R}^+\times \Lambda_0\to G$, where $\Lambda_0=H^{-1}(1/2)\subset T^*_0G$ is the initial cylinder of co-velocities in the cotangent space at the origin.  Recall that $T^*_0G$ was originally coordinated by $(h_0,h_i,w_i)$. Then we used coordinates $(\alpha,\rho,K_i,K^\perp_i)$, and finally, in Corollary \ref{geodesics_cor}, we defined parameter $\sigma=\rho|K^\perp|/|K|$ and then we wrote the equations for geodesics in an orthonormal basis $k=K/|K|$ and $k^\perp=K^\perp/|K^\perp|$. Moreover,  we substituted $|K|$ from the level set condition $(1+\sigma^2)/\rho^2=1/|K|^2$ and thus we obtained the parameterization  of cylinder $\Lambda_0$ by $(\alpha,\rho,\sigma,k_i,k^\perp_i)$, where  $\alpha\in S^1,$ $\rho>0,$ $\sigma\geq 0$ and $(k_i,k_i^\perp)$ lies in the Stiefel manifold of 2-frames in $\R^n$. It is easy to see that these coordinates are related to the  original coordinates on the cotangent space as follows
\begin{align*}
h_0=\frac{\sin\alpha}{\sqrt{1+\sigma^2}}, \quad h_i=k_i\frac{\cos\alpha}{\sqrt{1+\sigma^2}}+k_i^\perp\frac{\sigma}{\sqrt{1+\sigma^2}}, \quad w_i=k_i\frac{\rho}{\sqrt{1+\sigma^2}}.
\end{align*}
Indeed, we can verify that $h_0^2+h_1^2+\cdots+h_n^2=1.$ 
An explicit formula for the exponential map in these coordinates of $\Lambda_0$ reads
\begin{align}
\exp:\:  (t,\rho,\sigma,\alpha,k_i,k_i^\perp)\mapsto (x(t),\ell_1(t)k_i+\ell_2(t)k_i^\perp,y_1(t)k_i+y_2(t) k_i^\perp),
\end{align}
where  $x(t)\in \R$ and $\ell(t),y(t)\in\R^2$ are given by equation \eqref{geodesics_polar} or equivalently by \eqref{geodesics_greek}.
  Note that, considering $\ell(t),y(t)$ as curves in $\R^2$, we have already used the symmetry given by the choice of plane $\operatorname{span}\{k_i,k_i^\perp\}$. The factorization over the $SO(n)$ action now fixes the last freedom given by the choice of orthonormal basis in this plane. 
   In other words, we can always set $k_i$ to $e_1$ and $k_i^\perp$ to $e_2$ and thus $(\rho,\sigma,\alpha)$ form a set of coordinates of factorized cylinder $\widehat{\Lambda}_0:=\Lambda_0/SO(n).$ The invariants that form coordinates of the factor space $\widehat{G}:=G/SO(n)$  has been discussed above, see Proposition \ref{factor_spaces}.
   \begin{proposition} \label{factorized_exp}
      Let  $\widehat{\Lambda}_0=\Lambda_0/SO(n)$ with coordinates $(\rho,\sigma,\alpha)$ as described above, i.e. $\rho>0,\sigma\geq 0$ and $\alpha\in S^1.$ Then, the factorized exponential map is given by
           \begin{align*}
   \widehat{\exp}:\: \mathbb{R}^+\times\widehat{\Lambda}_0\to\widehat{G} \quad (t,\rho,\sigma,\alpha)\mapsto (x,\ell^2, y\cdot \ell,y\wedge \ell),
   \end{align*}
   where    
   \begin{align} \label{invariants}
   \begin{split} 
   x &=\frac{2}{\rho}  \sin \tau  \sin \! \left(\tau +\alpha \right), \\
   \ell^2 &=\frac{4}{\rho^2} \left(\sin^{2}\tau \cos^{2}\!\left(\tau+\alpha \right)+ \tau^{2} \sigma^{2}\right)
   , \\
   y\cdot \ell &=\frac{2}{\rho^3}\cos \! \left(\tau+\alpha \right)\left( \sin \tau  \left(\tau -\cos \tau  \sin \tau\right)+ 2\tau \sigma^{2}  \left(\sin\tau-\tau  \cos \tau \right)\right),\\
   |y\wedge\ell|& = \frac{2}{\rho^3}\sigma\left(  \tau  \left(\tau -\cos \tau  \sin \tau\right)- 2\sin\tau\cos^2 \! \left(\tau+\alpha \right)  \left(\sin\tau-\tau  \cos \tau \right)\right),
   \end{split}
   \end{align}
   where $\tau=t\rho/2\sqrt{1+\sigma^2}$. Alternatively, using the square of the length of vector $y$ as a coordinate of $\widehat{\Lambda}_0$, we get
   \begin{align} \label{invariants2}
   y^2=\frac{1}{\rho^4} \left(\left(\tau -\cos\tau \sin\tau\right)^{2}+ 4\sigma^{2} \cos^{2}\!\left(\tau +\alpha \right) \left(\sin\tau-\tau  \cos\tau\right)^{2}\right).
   \end{align}
   \end{proposition}
   \begin{proof}
      The result is obtained using equation \eqref{geodesics_greek} for geodesics.
   \end{proof}
  
\begin{remark}
  The geodesics in $G_1$ correspond exactly to the limiting case $\sigma=0.$ Indeed, it is easy to see that, in such a case,  the formula for  $\widehat{\exp}:(t,\rho,\alpha)\mapsto (x,|\ell|,|y|)$ coincides with the formula \eqref{geodesics_Heisenberg} for Heisenberg geodesics. The other invariants satisfy $y\wedge \ell=0$ and $y\cdot\ell=|y||\ell|$ in such a case.
\end{remark}

\subsection{Description of the cut locus and the cut time}
\label{cut_locus_description}
Let us discuss individually the types of geodesics given in Proposition \ref{characterization}. Obviously, the lines are optimal for all times. Next, the geodesics in $G_1$ are preimages of the Heisenberg geodesic in the factorization map and since the action of $SO(n)$ is a symmetry with respect to the sub-Riemannian metric (and thus the symmetry of the exponential map), they have the same length and they lose their optimality at the same time as in the Heisenberg case $\mathbb{H}_3$, see section \ref{Heisenberg}. It means that the geodesics in $G_1$ starting at the origin lose their optimality when they meet the vertical set $(0,0,y)\in\R\oplus\R^n\oplus\R^n $. These points are both Maxwell points and conjugate points, and the cut time is equal to
\begin{align} \label{cut_time_G1}
t_{cut} = \sqrt{4\pi|y|} = \frac{2\pi}{\rho},
\end{align}
where the second equality holds for our parameterization of geodesics as in \eqref{geodesics_polar} or \eqref{geodesics_greek}.
Indeed, substituting $\sigma=0$ and   $\tau_{cut}=\pi$ into these equations we get that the end-point of any geodesics in $G_1$ in such critical time has coordinates $(0,0,\pi k_i/\rho^2),$ where $k_i$ is an arbitrary unite vector in $\R^n$ and $\rho>0$. Changing parameter $\alpha$ we get an infinite number of geodesics of the same length from the origin to this point.

Now let us consider the geodesics in $G_2$, i.e., the geodesics with $\sigma\neq 0$, and let us look at their end-points in the same critical re-scaled time $\tau=\pi$. Substitution into formulas \eqref{invariants} and \eqref{invariants2} for the factorized exponential map gives
 \begin{align} \label{cut_locus_parametric}
\begin{split} 
x &=0, \\
|\ell| &=\frac{2\pi \sigma}{\rho} 
, \\
y\cdot \ell &=\frac{4\pi^2 \sigma^{2} }{\rho^3}\cos \! \left(\pi+\alpha \right),\\
|y\wedge\ell|& = \frac{2\pi^2\sigma}{\rho^3}, \\
|y|&=\frac{\pi}{\rho^2} \sqrt{1+ 4\sigma^{2} \cos^{2}\!\left(\pi +\alpha \right)}.
\end{split}
\end{align}
We observe that, in contrast to the Heisenberg case, the end-points depend on the parameter $\alpha$. It corresponds to the fact that the action of $SO(n)$ on $G$ does not produce Maxwell points, as mentioned earlier. Nevertheless, there is still one discrete symmetry that fixes these points, namely
\begin{align}
\alpha\mapsto -\alpha.
\end{align}
Indeed, $\cos \! \left(\pi+\alpha \right)=\cos \! \left(\pi-\alpha \right)$ and thus this symmetry acts trivially on the invariants \eqref{cut_locus_parametric} however it obviously does not act trivially on the projected geodesics \eqref{invariants}. In other words, we have found two different geodesics of the same length that meet at the same point. Hence, $\tau=\pi$ is the Maxwell re-scaled time also for geodesics in $G_2.$  This brings us to a conjecture that $\tau_{cut}=\pi$ for all geodesics. 
Indeed, we will prove this conjecture by the extended Hadamard technique in the next section. Thus, a crucial statement for the optimal synthesis on the path group reads as follows. Let us recall the relation $\tau=t\rho/2\sqrt{1+\sigma^2}$ between the original time and the rescaled time.
\begin{proposition} \label{cut_time_prop}
	The cut time of sub-Riemannian geodesics on the path group $G$ parametrized as  in Corollary \ref{geodesics_cor} by  $\theta=(\alpha,\rho,\sigma,k,k^\perp)\in\Lambda_0$  is equal to
\begin{align} \label{cut_time}
t_{cut}(\theta)=\frac{2\pi}{\rho}\sqrt{1+\sigma^2}.
\end{align}
For $\sigma=0$, the time $t_{cut}$ is also the first conjugate time. For $\sigma > 0$, it coincides with the first conjugate time only for $\alpha=0$ and $\alpha=\pi$. 
\end{proposition}
A proof of this proposition is given below, see Section \ref{proof}.
The substitution of this cut time into equations for geodesics immediately gives a parametric description of the cut locus. An explicit description of the cut locus and cut time is given in the following theorem, which is considered as the main result of this paper.	
	\begin{theorem} \label{thm}
	Let $G=\R\oplus\R^n\oplus\R^n$ be a Carnot group with the path geometry structure.
	The cut locus $\operatorname{Cut}_0$ of sub-Riemannian geodesics issuing from the identity is formed by points $(0,0,y)\in G$ and points $(0,\ell,y)\in G$ such that $\ell\wedge y\neq 0$ and 
	\begin{align} \label{cut_locus_condition}
	\pi |y|\cos^2\varphi \leq |\ell|^2\sin\varphi, 
	\end{align}
	where $\varphi$ is the angle between vectors $\ell,y\in\R^n$ and where $|\cdot|$ denotes the euclidean norm in $\R^n$.
	The corresponding cut time is given by
	\begin{align} \label{cut_time_cases}
	t_{cut}=\begin{cases}
	 \sqrt{4\pi|y|} & \text{ if } \; \ell\wedge y=0, \\
	 \sqrt{4\pi|y|\sin\varphi+|\ell|^2} & \text{ if } \; \ell\wedge y\neq 0.
	 \end{cases}
	\end{align}
	The points $(0,0,y)$ are also conjugate points, and they are reached by infinitely many optimal geodesics. The points $(0,\ell,y)$ in the cut locus are reached by two optimal geodesics, and they are also conjugate points if and only if  \eqref{cut_locus_condition} becomes equality. 
\end{theorem} 
\begin{proof}
Substituting the cut time from Proposition \ref{cut_time_prop} into equations \eqref{invariants} for the projected geodesics, we get parametrization \eqref{cut_locus_parametric} of the cut locus factorized by the action of $SO(n)$  by three parameters $\alpha\in S^1,\rho>0$ and $\sigma\geq 0$. 
In particular, we see that $x=0.$ The cut points satisfying $\ell\wedge y =0$ correspond to geodesics with parameter $\sigma=0$, i.e. to geodesics in $G_1$. In such a case the factorized cut locus has dimension one, namely, it is the Heisenberg cut locus $x=|\ell|=0$, $|y|=\pi/\rho^2$ and its preimage under the action of $SO(n)$ is the vertical space $(0,0,y)$, $y\in\R^n,$ as discussed above. We also get the Heisenberg cut time $t_{cut}=2\pi/\rho=\sqrt{4\pi|y|}$. 
If a cut point satisfies $\ell\wedge y \neq 0$,  we have $\sigma\neq 0$ and it corresponds to a geodesic in $G_2$. In such a case, the dimension of the cut locus is three and for the inverse to factorized exponential map, we may compute
	\begin{align*}
    \rho=\sqrt{\frac{\pi}{|y|\sin\varphi}}, \;
    \sigma=\frac{|\ell|}{\sqrt{4\pi |y| \sin\varphi}}, \;
    \cos\alpha = \frac{\cos\varphi}{|\ell|}\sqrt{\frac{\pi |y|}{\sin\varphi}}.
	\end{align*}
The inverse is well defined since $|\ell|,|y|,\varphi \neq 0$. From here, we see that a point $(0,\ell,y)\in G_2$ is in the cut locus if and only if the absolute value of the right-hand side of the last equation is not greater than one. This condition gives the inequality  \eqref{cut_locus_condition} from the theorem.  Moreover, substituting the formulas for parameters $\rho,\sigma$ into the equation \eqref{cut_time} for the cut time, we get formula \eqref{cut_time_cases} for the cut time in terms of invariants.
\end{proof}

\begin{remark}
	The inequality \eqref{cut_locus_condition} describes the cut locus in $G_2$ in terms of invariants $|\ell|,|y|,\varphi$. If we use $|\ell|,\ell\cdot y,\ell\wedge y$ as coordinates of factor space $G_2/SO(n)$ instead, the cut locus is described by inequality
	\begin{align} \label{cut_locus_condition2}
	\pi (\ell\cdot y)^2\leq |\ell|^3|\ell\wedge y|. 
	\end{align}
\end{remark}

Let us look at the cut locus in more detail. We observe that inequality \eqref{cut_locus_condition} is satisfied for all $|\ell|,|y|$ if $\varphi=\pi/2$ and that the inequality holds for $\varphi$ if and only if it holds for $\pi-\varphi.$ Moreover, for $\varphi\in\langle 0,\pi/2\rangle$, the left-hand side is strictly decreasing to zero and the right-hand side is strictly increasing from zero. Thus for fixed lengths of vectors $\ell,y$ there exist a unique $0<\varphi_0<\pi/2$ such that \eqref{cut_locus_condition} is satisfied for $\varphi\in\langle\varphi_0,\pi-\varphi_0\rangle$, namely
\begin{align} \label{phi_0}
	\varphi_0=\arcsin\left( \frac{-|\ell|^2+\sqrt{|\ell|^4+4\pi^2|y|^2}}{2\pi|y|}\right).
\end{align}
The situation in a fixed plane in $\R^n$ is depicted in Figure \ref{cut_locus_visualization} on the left. 
The points $(0,\ell,y)$ belong to the cut locus if the angle $\varphi$ between defining vectors $\ell,y\in \R^n$  is "right enough." 
\begin{figure}[ht] 
	  \begin{tikzpicture}[scale=0.7]
	\draw[->,arrows={-latex},semithick] (0,2) -- (3,7/5)
	node[near end,below]{\scriptsize $\ell$};
	\draw[dashed] (-0.3,0.5) -- (1,7);
	\draw[dashed] (-4,2.8) -- (4,1.2);
	
	\draw[->,arrows={-latex},semithick] (0,2) -- (4.50,5.97) node[near end, below]{\scriptsize $y$};
	\draw[->,arrows={-latex},semithick] (0,2) -- (-2.62,7.40);
	
	\draw[->,arrows={-latex},semithick] (0,2) -- (-1.48,3.35);
	\draw[->,arrows={-latex},semithick] (0,2) -- (1.88,2.67) node[near end, below]{\scriptsize $y$}; 
	
	\draw[<->,arrows={latex-latex},thick,color=blue!50] (-1.31,4.7) arc (115.88:41:3) node[above,yshift=0.8cm,color=blue]{\scriptsize $\varphi\in\langle\varphi_0,\pi-\varphi_0\rangle$};
	\draw[<->,arrows={latex-latex},thick,color=blue!50] (-0.74,2.675) arc (137.63:19.23:1); 
	\end{tikzpicture}
  \begin{tikzpicture}[scale=0.8]

\begin{axis}[
axis x line=center,
axis y line=center,
xtick={0,...,0},
ytick={0,...,0},
xticklabel=\empty,
yticklabel=\empty,
xlabel={\scriptsize $|\ell|$},
ylabel={\scriptsize $|y|$},
xlabel style={below right},
ylabel style={above right},
xmin=0,
xmax=2.5,
ymin=0,
ymax=2.5]





\addplot [domain=0:3, samples=100, name path=f0, thick, color=blue!50]
{7.96*x^2};

\addplot [domain=0:3, samples=100, name path=g0, thick, color=blue!50]
{8.12*x^2};

\addplot[blue!10, opacity=0.6] fill between[of=f0 and g0, soft clip={domain=0:2}];

\node[color=blue, font=\scriptsize] at (axis cs: 0.9,2.3) {$\sigma=0.1$};

\addplot [domain=0:3, samples=100, name path=f1, thick, color=blue!50]
{0.884*x^2};

\addplot [domain=0:3, samples=100, name path=g1, thick, color=blue!50]
{1.031*x^2};

\addplot[blue!10, opacity=0.6] fill between[of=f1 and g1, soft clip={domain=0:2}];

\node[color=blue, font=\scriptsize] at (axis cs: 1.9,2.3) {$\sigma=0.3$};

\addplot [domain=0:2.5, samples=100, name path=f2, thick, color=blue!50]
{0.1243*x^2};

\addplot [domain=0:2.5, samples=100, name path=g2, thick, color=blue!50]
{0.2346*x^2};

\addplot[blue!10, opacity=0.6] fill between[of=f2 and g2, soft clip={domain=0:2.5}];

\node[color=blue, font=\scriptsize] at (axis cs: 2.1,1.45) {$\sigma=0.8$};

\addplot [domain=0:2.5, samples=100, name path=f3, thick, color=blue!50]
{0.0088*x^2};

\addplot [domain=0:2.5, samples=100, name path=g3, thick, color=blue!50]
{0.0538*x^2};

\addplot[blue!10, opacity=0.6] fill between[of=f3 and g3, soft clip={domain=0:2.5}];

\node[color=blue, font=\scriptsize] at (axis cs: 2.3,0.45) {$\sigma=3$};

\end{axis}
\end{tikzpicture}
\caption{Visualizations of the cut locus. Left: A sketch of vectors $\ell, y$ such that point $(0,\ell,y)$ lies in the cut locus; displayed for $|\ell|=3,|y|=2$ and $|\ell|=3,|y|=6.$ Right: Cross sections of the factorized cut locus by planes $\sigma=\text{const}$.}
\label{cut_locus_visualization}
\end{figure}
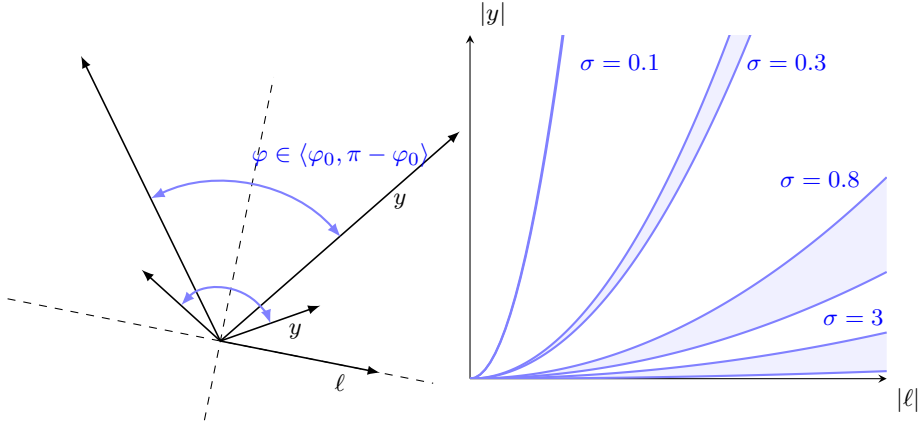

It is illustrative to look also at the cut locus in terms of parameters $\ell,y$ and $\sigma$ instead of $\varphi$ since  this parametrization allows to visualize the Heisenberg limit $\sigma\to 0.$ 
From the second and the last equation in parametric description \eqref{cut_locus_parametric} of the factorized cut locus we get $\rho=2\pi\sigma/|\ell|$ and  $\pi/\rho^2\leq|y|\leq \pi/\rho^2\sqrt{1+4\sigma^2}$ respectively. Substituting  parameter $\rho$ in this inequality gives
\begin{align*}
	\frac{1}{4\pi\sigma^2}|\ell|^2\leq |y|\leq \frac{\sqrt{1+4\sigma^2}}{4\pi\sigma^2}|\ell|^2,
\end{align*}
hence the cross-section of the cut locus defined by a fixed $\sigma$ is the two-dimensional area in plane $x=0$ between parabolas centered in the origin with different focal lengths, see Figure \ref{cut_locus_visualization} on the right.
In particular, we see that for $\sigma \to 0$ we get the Heisenberg cut locus $(0,0,|y|).$ On the other hand, we get $y\to 0$ if $\sigma\to \infty.$


\subsection{Proof of Proposition \ref{cut_time_prop}} \label{proof}
Since the Heisenberg case $\sigma=0$ has already been discussed, we may restrict to initial co-velocities in $\Lambda_0$ with $\sigma>0$, i.e., we may restrict to geodesics in $G_2$. Moreover, it is sufficient to study the factorization of the exponential map by the $SO(n)$ symmetry, $\widehat{\exp}:\: \mathbb{R}^+\times\widehat{\Lambda}_0\to\widehat{G}_2$, where we denote $\widehat{G}_2=G_2/SO(n)$. Note that applying these symmetries reduces the dimension of initial and target manifolds from the original $2n+1$ to 4. We will apply the extended Hadamard technique and follow the steps in \cite{ABB}. Namely, we define the open sets
\begin{align*}
N_1=\{t\theta\;|\; \theta\in\widehat{\Lambda}_0, t\in (0,t_{cut}(\theta))\}\subset T_0^*\widehat{G}_2, \quad N_2= \widehat{\exp}(N_1)\subset \widehat{G}_2
\end{align*}
where $t_{cut}(\theta)$ is the conjectured cut time given by \eqref{cut_time}, and where by $\widehat{\exp}(\lambda)$ for $\lambda\in N_1$ we mean $\widehat{\exp}(t,\theta)$ where $\lambda=t\theta$, see the explanation in \ref{SRgeodesics}. First, we prove that $N_2=\widehat{G}_2\setminus \widehat{\text{Cut}}_0$, where $\widehat{\text{Cut}}_0$ is the factorization of the conjectured cut locus. Then we show that $\widehat{\exp}|_{N_1}$ is regular and proper, which allows us to conclude that it is a covering. In the last step, we prove that the degree of the covering is one by showing that the induced map on the fundamental groups is surjective. At the end of this procedure, we conclude that $\widehat{\exp}$ is a global diffeomorphism between $N_1$ and $N_2$. As a consequence, the conjectured cut time is the true one. 

We will work with the re-scaled time $\tau$ during the proof. This is possible since the change between times $t=2\tau\sqrt{1+\sigma^2}/\rho$ is a diffeomorphism. So we may view $N_1$ as a four-dimensional open half-cylinder without its central axis, namely we have $N_1=(0,\pi)\times S^1\times \mathbb{R}^+\times\mathbb{R}^+$ parameterized by $(\tau,\alpha,\rho,\sigma)$.  

\subsection*{Step 1}
We prove the equality of sets $N_2=\widehat{G}_2\setminus \widehat{\text{Cut}}_0$ by showing that both inclusions hold. Let us recall that the factorized cut locus $\widehat{\text{Cut}}_0$ is the image of $\widehat{\exp}$ in the critical rescaled time $\tau_{cut}=\pi$. Let us also recall that  we may use the set of invariants $x,\ell^2,\ell\cdot y, \ell\wedge y$ or invariants $x,\ell^2,\ell\cdot y, y^2$ as coordinates on $\widehat{G}_2$ since the transition between the coordinates is regular for $\ell\wedge y\neq 0.$
\begin{lemma*}
    $N_2\supseteq\widehat{G}_2\setminus \widehat{\text{Cut}}_0$, i.e. $\widehat{\exp}|_{N_1}$ is surjective on $\widehat{G}_2\setminus \widehat{\text{Cut}}_0.$ In other words, each point in $\widehat{G}_2$ that is not in the cut locus can be reached by a (factorized) geodesics in a time smaller than $\tau_{cut}=\pi.$
\end{lemma*}
\begin{proof}
     Let $(x,\ell^2,\ell\cdot y, y^2)\in \widehat{G}_2$ be a point that 
     does not lie in $\widehat{Cut}_0$. We will show that then the equations for $\widehat{\exp}$ given in Proposition \ref{factorized_exp}  have a solution $(t,\rho,\sigma,\alpha)$. At first, let us compute the Jacobian $J_3$ of the function of three variables $(\rho,\sigma,\alpha)\mapsto (x,\ell^2,\ell\cdot y)$ given by equation \eqref{invariants}. A direct computation gives
     \begin{align*}
         J_3&=-\frac{32}{\rho^7}\sigma\tau\sin\tau\Big(2\sigma^2\tau^2(\sin\tau-\tau\cos\tau)+\cos^2(\tau+\alpha)(\tau^2+\tau\sin\tau\cos\tau-2\sin^2\tau)\\
         &+\tau\sin\tau(\tau-\sin\tau\cos\tau)\Big).
     \end{align*}
    Since $\rho,\sigma>0$ and all the functions $\sin\tau-\tau\cos\tau,$ $\tau^2+\tau\sin\tau\cos\tau-2\sin^2\tau$, $\tau-\sin\tau\cos\tau$ are positive for $\tau\in(0,\pi)$, we conclude that 
    \begin{align*}
        J_3(t,\rho,\sigma,\alpha)<0 \qquad \text{for}  \qquad\tau\in(0,\pi).
    \end{align*}
    Hence, by the inverse function theorem,  we can express parameters $\rho(\tau),\sigma(\tau),\alpha(\tau)$ as functions of $\tau$ from the first three equations in \eqref{invariants} and substitute them into the fourth equation \eqref{invariants2}. Now it is sufficient to prove that this single equation of one variable $\tau$ has a solution between zero and $\pi$. Namely, we need to prove that there is always such a $\tau$ that satisfies
    \begin{align*} 
   F_{y^2}(\tau):=y^2\rho^4-\left(\left(\tau -\cos\tau \sin\tau\right)^{2}+ 4\sigma^{2} \cos^{2}\!\left(\tau +\alpha \right) \left(\tau  \cos\tau-\sin\tau\right)^{2}\right)=0.
   \end{align*}
   We prove the existence by exploring the local behavior of this function at the limiting points 0 and $\pi$. 
   
   First, let us assume that $x\neq 0.$ Then we can conclude that $\rho(0)=\rho(\pi)=0$ 
  since, by substituting the limiting points into the first equation in \eqref{invariants}, we get  $\rho(0)x=\rho(\pi)x=0$. Therefore 
   $F_{y^2}(0)=y^2\rho(0)^4=0$ and not just that, also the derivatives up to order three will vanish in $\tau=0.$ Then, by using the Taylor expansion in zero, we get 
   $$F_{y^2}(\tau)=\big(\rho'(0)\big)^4y^2\tau^4+o(\tau^4)>0$$
   for $\tau$ small enough since $\rho'(0)\neq 0$. This can be shown by computing  $x^2+\ell^2$ from the first two equations in $\eqref{invariants}$ and by taking two derivatives. Namely, we compute $\rho^2(x^2+\ell^2)=4(\sin^2\tau+\tau^2\sigma^2),$ and by taking the second derivative evaluated in zero, we obtain $(\rho'(0))^2(x^2+\ell^2)=4(1+\sigma^2).$ On the other hand, substituting the other end of our interval, $\tau=\pi$, into the function, we get
   $$F_{y^2}(\pi)=-\pi^2\big(1+4\sigma(\pi)^2\cos^2(\alpha(\pi))\big)<0$$
   since $\rho(\pi)=0$ as discussed above. The claim then follows by the continuity of function $F_{y^2}.$

   Now let us discuss the case $x=0.$ Then the first equation in \eqref{invariants} implies that $\sin(\tau+\alpha)=0,$ it means $\cos(\tau+\alpha)=\pm 1,$ and the formula for $\widehat{\exp}$ simplifies to
   \begin{align*}
     x &=0, \\
   \ell^2 &=\frac{4}{\rho^2} \left(\sin^{2}\tau + \tau^{2} \sigma^{2}\right)
   , \\
   y\cdot \ell &=\frac{2}{\rho^3}\left( \sin \tau  \left(\tau -\cos \tau  \sin \tau\right)+ 2\tau  \,\sigma^{2}  \left(\sin\tau-\tau  \cos \tau \right)\right),\\
   y^2&=\frac{1}{\rho^4} \left(\left(\tau -\cos\tau \sin\tau\right)^{2}+ 4\sigma^{2} \left(\tau  \cos\tau-\sin\tau\right)^{2}\right).
   \end{align*}
Computing the Jacobian $J_2$ of function of two variables $(\rho,\sigma)\mapsto(\ell^2,\ell\cdot y)$ we get 
\begin{align*}
    J_2=\frac{16\sigma\tau}{\rho^6}\Big(2\sigma^2\tau^2(\sin\tau-\tau\cos\tau)+\sin\tau(3\tau^2+\tau\cos\tau\sin\tau-4\sin^2\tau)\Big).
\end{align*}
It is easy to see that $J_2>0$ for $\tau\in(0,\pi)$ since the same is true for functions $\sin\tau-\tau\cos\tau$ and $3\tau^2+\tau\cos\tau\sin\tau-4\sin^2\tau.$ Therefore, we may express $\rho(\tau),\sigma(\tau)$ as functions of time $\tau$ and substitute them into the equation for $y^2.$ Similarly as before, we want to show the existence of a solution of 
\begin{align*}
    F_{y^2}(\tau):=y^2\rho^4-\Big((\left(\tau -\cos\tau \sin\tau\right)^{2}+ 4\sigma^{2} \left(\tau  \cos\tau-\sin\tau\right)^{2}\Big)=0
\end{align*}
in $(0,\pi).$ From the second equation we get $\ell^2\rho(0)^2=0$ and thus $\rho(0)=0.$ We still have the equality $\rho^2(x^2+\ell^2)=4(\sin^2\tau+\tau^2\sigma^2)$ that yields $\rho'(0)\neq 0.$ It means that, in $\tau=0$, we have the same expansion of $F_{y^2}$ as before. In particular, it is positive near zero. In $\tau=\pi$, we get
\begin{align*}
    F_{y^2}(\pi)=y^2\rho(\pi)^4-\pi^2(1+4\sigma(\pi)^2),
\end{align*}
and we show that this value is negative for $y\not\in Cut_0.$ Then we can argue again by the continuity of this function that a solution between 0 and $\pi$ exists. Indeed, substituting $\tau=\pi$ in the above equations for $\widehat{\exp}$ and solving for $\rho,\sigma$ gives 
\begin{align*}
    \rho(\pi)=\frac{\ell^2}{\ell\cdot y}, \qquad \sigma(\pi)=\frac{\ell^3}{2\pi(\ell\cdot y)},
\end{align*}
and thus $$F_{y^2}(\pi)=\frac{y^2\ell^8-\pi^2(\ell\cdot y)^4-\ell^6(\ell\cdot y)^2}{(\ell\cdot y)^4}=\frac{\ell^6|\ell\wedge y|^2-\pi^2(\ell\cdot y)^4}{(\ell\cdot y)^4},$$
where we have used $\ell^2y^2-(\ell\cdot y)^2=|\ell\wedge y|^2.$ Now we see that $F_{y^2}(\pi)<0$ is opposite to inequality \eqref{cut_locus_condition2} describing the points in the cut locus. Let us remark that it is also possible to derive this inequality from the parametric description of the cut locus given in \eqref{cut_locus_parametric}.
\end{proof}

For the other inclusion, we basically need to prove that $F_{y^2}(\tau)>0$ for all $\tau\in(0,\pi)$ if the invariants lie in the cut locus. Obviously,  local information in the limiting points is not enough and we must proceed differently, by using several estimates for this function. 

\begin{lemma*}
    $N_2\subseteq\widehat{G}_2\setminus \widehat{\text{Cut}}_0$, i.e. $\widehat{\exp}|_{N_1}$ does not have values in $\widehat{\text{Cut}}_0$. In other words, it is not possible to reach any point in the cut locus by a time smaller than $\tau_{cut}=\pi.$
\end{lemma*}
\begin{proof}
We will prove that equations for $\widehat{\exp}$ do not have any solution in $N_1$ if the invariants lie in $\widehat{Cut}_0$. Let us consider a point in the cut locus and its parametric description \eqref{cut_locus_parametric}  by parameters $\bar{\rho},\bar{\sigma},\bar{\alpha}$.  Then  equation  \eqref{invariants} for $\widehat{\exp}$  in coordinates $(x,\ell^2,|\ell\wedge y|,y^2)$ reads
    \begin{align*}
 \frac{2}{\rho}  \sin \tau  \sin \! \left(\tau +\alpha \right)=&0, \\
\frac{4}{\rho^2} \left(\sin^{2}\tau \cos^{2}\!\left(\tau+\alpha \right)+ \tau^{2} \sigma^{2}\right)=&\frac{4\pi^2 \bar{\sigma}^2}{\bar{\rho}^2}, \\
\frac{2}{\rho^3}\sigma\left(  \tau  \left(\tau -\cos \tau  \sin \tau\right)- 2\sin\tau\cos^2 \! \left(\tau+\alpha \right)  \left(\sin\tau-\tau  \cos \tau \right)\right)=&\frac{2\pi^2\bar{\sigma}}{\bar{\rho}^3}, \\
\frac{1}{\rho^4} \left(\left(\tau -\cos\tau \sin\tau\right)^{2}+ 4\sigma^{2} \cos^{2}\!\left(\tau +\alpha \right) \left(\sin\tau-\tau  \cos\tau\right)^{2}\right)=&\frac{\pi^2}{\bar{\rho}^4} (1+ 4\bar{\sigma}^{2} \cos^{2}\bar{\alpha}).
\end{align*}
We immediately see from the first equation that we must have $\alpha=-\tau$ if a solution satisfying $\tau\in(0,\pi)$ would exist. Moreover, by changing the scaling factor $\rho$, we can set parameter $\bar{\rho}$ on the right-hand side to any value. In other words, a solution exists if and only if it exists for $\bar{\rho}=1.$ Therefore, it is sufficient to prove that the following system with variables $\tau,\rho,\sigma$ does not have solution for any choice of parameters $\bar{\sigma},\bar{\alpha}$
    \begin{align*}
F_{\ell^2}&:=\frac{1}{\rho^2} \left(\sin^{2}\tau+ \tau^{2} \sigma^{2}\right)-\pi^2 \bar{\sigma}^2=0, \\
F_{\ell\wedge y}&:=\frac{1}{\rho^3}\sigma\left(  \tau  \left(\tau -\cos \tau  \sin \tau\right)- 2\sin\tau \left(\sin\tau-\tau  \cos \tau \right)\right)-\pi^2\bar{\sigma}=0, \\
F_{y^2}&:=\frac{1}{\rho^4} \left(\left(\tau -\cos\tau \sin\tau\right)^{2}+ 4\sigma^{2} \left(\sin\tau-\tau  \cos\tau\right)^{2}\right)-\pi^2 (1+ 4\bar{\sigma}^{2} \cos^{2}\bar{\alpha})=0.
\end{align*}
Similar to the proof in the previous lemma, we will assume that the first two equations are satisfied and we will prove that the last equation cannot be satisfied.
By leaving out term $\sin^2\tau$ in the first equation, we get an estimate $\pi\bar{\sigma}\rho\geq \tau\sigma$. Using this inequality in the second equation, we get an estimate for the value of parameter $\rho$ in terms of $\tau$
\begin{align*}
    \frac{1}{\rho^2}\geq\frac{\pi\tau}{f_\wedge(\tau)},
\end{align*}
where
\begin{align*}
   f_\wedge(\tau)&= \tau  (\tau -\cos \tau  \sin \tau)- 2\sin\tau (\sin\tau-\tau  \cos \tau) \\
   &=\tau^2+\tau\sin\tau\cos\tau-2\sin^2\tau>0 
\end{align*}
for $\tau\in(0,\pi)$. The latter inequality can be easily shown by computing the second derivative $f_\wedge''(\tau)=4\sin\tau(\sin\tau-\tau\cos\tau)>0,$ and by evaluating $f_\wedge(0)=0$. Now we use this estimate for $\rho$ to prove that $F_{y^2}>0$. Before we do this, we eliminate parameter $\sigma$ in $F_{y^2}$ using the equation $F_{\ell^2}=0$, namely we substitute $$\sigma=\frac{\pi^2\bar{\sigma}^2\rho^2-\sin^2\tau}{\tau^2},$$
and collecting terms with respect to powers of $\rho$ we obtain 
\begin{align*}
   F_{y^2}&= \frac{1}{\rho^4} \left(\left(\tau -\cos\tau \sin\tau\right)^{2}-\frac{4\sin^2\tau}{\tau^2}\left(\sin\tau-\tau  \cos\tau\right)^{2}\right)\\
   &+\frac{1}{\rho^2}\frac{4\pi^2\bar{\sigma}^2}{\tau^2}\left(\sin\tau-\tau  \cos\tau\right)^{2}-\pi^2 (1+ 4\bar{\sigma}^{2} \cos^{2}\bar{\alpha}).
\end{align*}
Then using the estimate and we get
\begin{align*}
 F_{y^2}&\geq \frac{\pi^2\tau^2}{f_\wedge(\tau)^2} \left(\left(\tau -\cos\tau \sin\tau\right)^{2}-\frac{4\sin^2\tau}{\tau^2}\left(\sin\tau-\tau  \cos\tau\right)^{2}\right)\\
   &+\frac{4\pi^3\bar{\sigma}^2}{f_\wedge(\tau)\tau}\left(\sin\tau-\tau  \cos\tau\right)^{2}-\pi^2 (1+ 4\bar{\sigma}^{2} \cos^{2}\bar{\alpha})\\
   &= \frac{\pi^2}{f_\wedge(\tau)^2}\left(\tau^2(\tau -\cos\tau \sin\tau)^2-4\sin^2\tau(\sin\tau-\tau  \cos\tau)^2-f_\wedge(\tau)^2\right)\\
   &+\bar{\sigma}^2\frac{4\pi^2}{f_\wedge(\tau)\tau}\left(\pi(\sin\tau-\tau  \cos\tau)^2-f_\wedge(\tau)\tau\cos^2\bar{\alpha}\right),
\end{align*}
where we have collected the terms with respect to the power of $\bar{\sigma}$. Since this parameter can be arbitrary, we must prove that both coefficients are positive for $\tau\in (0,\pi).$ Indeed, the absolute term is equal to 
\begin{align*}
  \frac{4\pi^2\sin\tau}{f_\wedge(\tau)} (\sin\tau-\tau  \cos\tau)
\end{align*}
by definition of the function $f_\wedge$, and it is positive since all factors are positive on the given interval. To prove that the other coefficient is positive, we apply the estimates  $\pi>\tau$ and  $\cos^2\bar{\alpha}<1$. Then, using the definition of $f_\wedge,$ we get
\begin{align*}
    \pi(\sin\tau-\tau  \cos\tau)^2-f_\wedge(\tau)\tau\cos^2\bar{\alpha}&>\tau\left((\sin\tau-\tau  \cos\tau)^2-f_\wedge(\tau)\right)\\
    &\quad =\tau\sin\tau\left(3\sin\tau-3\tau\cos\tau-\tau^2\sin\tau\right)>0.
\end{align*}
The positivity of the function in brackets follows from the fact that it vanishes in $\tau=0$ and that its derivative is equal to $\tau(\sin\tau-\tau\cos\tau)$ and thus is positive for $\tau\in(0,\pi)$. We conclude that $F_{y^2}>0$, which means that there is no solution of equations for $\widehat{\exp}$ such that $\tau$ lies in this interval. In other words, there is no geodesics that would reach  $\widehat{Cut}_0$ in time smaller than $\tau_{cut}=\pi$.
\end{proof}

\subsection*{Step 2}
We prove that the differential of $\widehat{\exp}: N_1\to N_2=\widehat{G}_2\setminus\widehat{Cut}_0$ is invertible at every point in $N_1,$ i.e. there are no conjugate points in $N_2$ for $\widehat{\exp}|_{N_1}.$ Recall that $N_1$ is parameterized by $t,\rho,\sigma,\alpha$ such that $\rho,\sigma>0,\alpha\in S^1$ and  $t\in(0,t_{cut})$, where $t_{cut}=2\pi\sqrt{1+\sigma^2}/\rho.$
\begin{lemma*}
	For any $\theta=(\rho,\sigma,\alpha)\in\widehat{\Lambda}_0$ the first conjugate time $t_{conj}$ satisfies 
	\begin{align*}
	t_{conj}(\theta)\geq t_{cut}(\theta).
	\end{align*}
 The equality holds if and only if  $\alpha=0$ or $\alpha=\pi.$
\end{lemma*}
\begin{proof}
We will prove that the Jacobian of $\widehat{\exp}|_{N_1}$ vanishes for $\alpha=0,\pi$ in time $t=t_{cut}$ and is  strictly positive otherwise. We will work with the formula for the factorized exponential map $\widehat{\exp}$ given in Proposition \ref{factorized_exp}. In particular, we will use the re-scaled time $\tau=t\rho/2\sqrt{1+\sigma^2}$. Note that this does not affect the sign of Jacobian since it is a change of a variable by a positive multiple. Then, by direct computation of the determinant of partial derivatives of the function $\widehat{\exp}$ given by $\eqref{invariants}$ that has been checked by a symbolic software, we find that the Jacobian can be written as
	\begin{align*}
	J_4=\frac{256\sigma}{\rho^9} f(\tau,\alpha)\left(A(\tau,\alpha)\sigma^2+B(\tau,\alpha)\right),
	\end{align*}
where $f(\tau,\alpha)$ is the positive function given by equation \eqref{f_det}  defining the determinant $\operatorname{det}(y,\ell)$, see lemma \ref{lema}, and where 
	\begin{align*}
	A(\tau,\alpha)&=\tau^2(\sin\tau-\tau\cos\tau)^2-(\tau^2-\sin^2\tau)^2\cos^2(\tau+\alpha), \\
	B(\tau,\alpha)&=\frac12\sin\tau(\sin\tau-\tau\cos\tau)f(\tau,\alpha). \label{B}
	\end{align*}
 Since we have proved already that function $f(\tau,\alpha)$ is strictly positive for all $\tau>0$ in the proof of lemma \ref{lema}, we know that the zeros of Jacobian $J$ are given by zeros of function $$A(\tau,\alpha)\sigma^2+B(\tau,\alpha).$$
 We show that $A(\tau,\alpha)>0$, $B(\tau,\alpha)>0$ and thus the whole Jacobian is positive for $\tau\in(0,\pi),$ which is equivalent to $\tau_{conj}\geq \pi=\tau_{cut}.$ Moreover, we see that the equality happens if and only if $A(\pi,\alpha)=B(\pi,\alpha)=0,$ which is equivalent to $\alpha=0$ or $\alpha=\pi.$ 
 
 The inequality $B(\tau,\alpha)>0$ follows directly from the defining formula for the function $B(\tau,\alpha)$. To prove $A(\tau,\alpha)>0$, we use the estimate $\cos^2(\tau+\alpha)\leq 1.$ Then we get 
 \begin{align*}
    A(\tau,\alpha) > A_1(\tau):=\tau^2(\sin\tau-\tau\cos\tau)^2-(\tau^2-\sin^2\tau)^2 > 0
 \end{align*}
 for $\tau\in(0,\pi).$ The last inequality can be proven as follows. Since $A_1$ is a difference of squares, it is a product of two factors, one of which is evidently positive on this interval, and the other is equal to
	\begin{align*}
	A_1^-(\tau)&:=\tau \sin\tau-\tau^2\cos\tau-\tau^{2}+\sin^{2}\tau >0.
	\end{align*}
  To see that this inequality  holds we rewrite the function $A_1^-$ by using well-known trigonometric formulae for double angles, namely
 $$
 A_1^-(2\tau)=4 \cos\tau(\cos\tau \sin^{2}\tau + \tau\sin\tau - 2\tau^{2}\cos\tau).
 $$
 The positivity of $A_1^-$ on $\tau\in(0,\pi)$ follows by estimating the function in brackets from below by its Taylor polynomial of degree twelve in $\tau=0$. The polynomial is positive up to its first root, which is approximately $2.75>\pi/2$.
\end{proof}

\begin{remark}
    A finer inspection of the first zeros and signs of functions  $A(\tau,\alpha)$, $B(\tau,\alpha)$ gives also an upper estimate for the conjugate (re-scaled) time, namely 
    $$
    \pi=\tau_{cut}\leq\tau_{conj} < \tau_0,
    $$
    where $\tau_0\approx 1,43\pi$ is the least positive solution of equation $\tan(\tau)=\tau$. Indeed,  it is easy to see that  $A(\tau,\alpha)\leq A_0(\tau):=\tau^2(\sin\tau-\tau\cos\tau)^2$, and thus the first zero $\tau_A$ of function $A$ viewed as a function of $\tau$ lies between $\pi$ and $\tau_0.$ It means that the continuous function $\sigma(\tau,\alpha)=\sqrt{-B/A}$ satisfies $\sigma(\pi,\alpha)=0$ and  also $\lim_{\tau\to\tau_A}\sigma(\tau,\alpha)=\infty.$ The other way around, for each $\sigma,$ there exists $\tau_{conj}\in(\pi,\tau_A)$ such that $A\sigma^2+B=0$. Therefore, $\pi<\tau_{conj}<\tau_A<\tau_0,$ as depicted in Figure \ref{AB_graphs}.
    \end{remark}
		\begin{figure}[ht]
	\begin{center}
		\includegraphics[height=47mm]{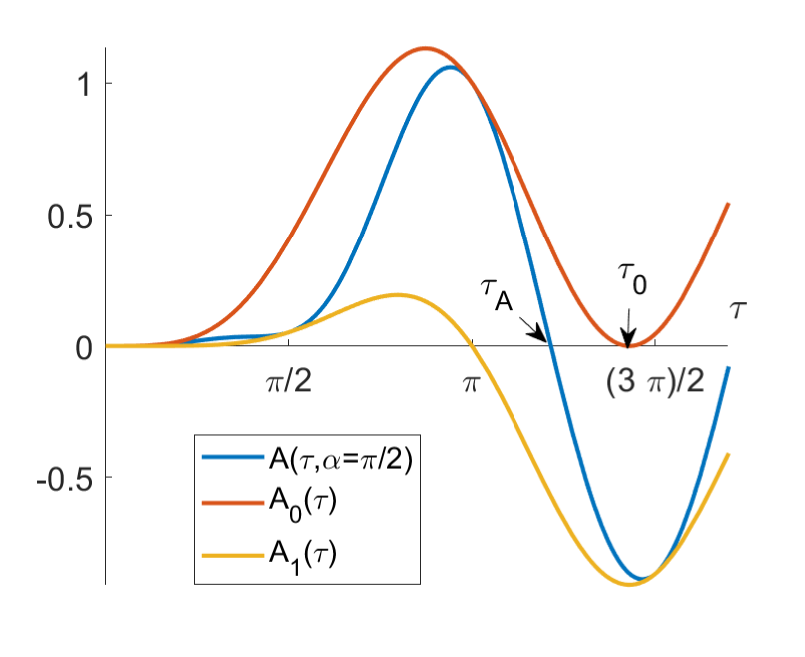}
		\includegraphics[height=47mm]{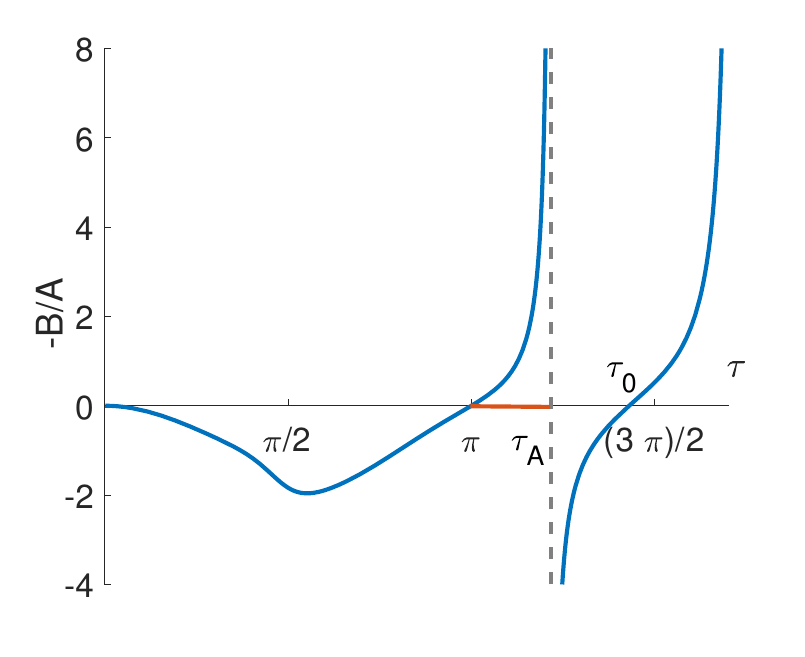}
			\caption{Left: Estimation of function $A(\tau,\alpha)$ (functions plotted for $\alpha=\pi/2$). Right: Graph of function $\sigma^2(\tau,\alpha)=-B/A(\tau,\alpha)$. For each $\sigma > 0$ there is a $\tau_{conj}\in(\pi,\tau_A)$ such that $-B/A(\tau_{conj})=\sigma^2$, i.e. $J(\rho,\alpha,\sigma,\tau_{conj})=0$. 
		}
		\label{AB_graphs}
	\end{center}
\end{figure}

\subsection*{Step 3}
We will prove that the map $\widehat{\exp}: N_1\to N_2$ is proper. Concretely, we will show that each sequence escaping from $N_1$ maps under $\widehat{\exp}$ to a sequence escaping from $N_2,$ by inspecting the equations $\eqref{invariants}$ and \eqref{invariants2}. Let us go through all types of such sequences.
\begin{itemize}
    \item for $\tau\to \tau_{cut}=\pi$  all points $(\pi,\rho,\sigma,\alpha)\in N_1$ are mapped to $\widehat{Cut}_0$ by the definition of the cut locus, and thus do not belong to $N_2$.
    \item for $\sigma\to \infty$ we get $|\ell|\to \infty$ unless we have also $\rho\to\infty.$ In such a case we get $x\to 0$, and then it depends on how fast tends the parameter $\rho$ to infinity relative to $\sigma$. If $\sigma/\rho\to \infty$, we still have $|\ell|\to \infty$.
    If $\sigma/\rho\to 0$, then the image of such sequence tends to the origin and thus is not in $N_2$.  If $\sigma/\rho\to L < \infty$, then $|\ell|\to 2\tau L$ but we have $|y|\to 0$ and $\varphi\to 0.$ Such a point does not even lie in $G_2$ (it is a point reachable by the linear type of geodesics, see Proposition \ref{characterization}).  
    \item for $\rho\to \infty$ the image tends to the origin unless also $\sigma\to \infty$ but this case has been already discussed.
    \item for $\rho\to 0$ we have $x\to\infty.$
\end{itemize}

\subsection*{Step 4}
The previous steps imply that $\widehat{\exp}|_{N_1}$ is a covering. We prove that it is a diffeomorphism, i.e., a one-fold covering, by showing that the induced map on the fundamental groups $[\widehat{\exp}]:\: \pi_1(N_1)\to\pi_1(N_2)$ is not only injective but it is also surjective. Actually, we show that both groups are isomorphic to $\mathbb{Z}$ and that a generator is mapped to a generator.

The fact that $\pi_1(N_1)=\mathbb{Z}$ is obvious since $N_1$ can be viewed as a four-dimensional open half-cylinder without its central axis, $N_1=(0,\pi)\times S^1\times \mathbb{R}^+\times\mathbb{R}^+$. The target space is described as $N_2=\widehat{G}_2\setminus \widehat{\operatorname{Cut}}_0,$ where $\widehat{G}_2=\mathbb{R}\times\mathbb{R}^+\times\mathbb{R}^+\times (0,\pi)$ in the parameterization by $(x,|\ell|,|y|,\varphi)$, and $\widehat{\text{Cut}}_0$ is formed by points in the subspace $\{x=0\}$  such that inequality \eqref{cut_locus_condition} holds. This inequality says that the angle $\varphi$ lies in an interval around the right angle, concretely  $\varphi\in \langle\varphi_0,\pi-\varphi_0\rangle$, where $\varphi_0$ is defined in terms of $|\ell|,|y|$ by \eqref{phi_0}, see the description of the cut locus in section \ref{cut_locus_description}. Hence the space $N_2$ is a union of the simply connected half-spaces $\{x>0\}$, $\{x<0\}$ and  two connected components of the subspace $\{x=0\}$. These components can be recognized by the value of $\varphi$ - we have either $\varphi<\pi/2$ or $\varphi>\pi/2$. It is easy to see that the homotopy class of a loop in $N_2$ depends only on transversal intersections with the subspace $\{x=0\}$. Moreover, if we have two consequent intersections that happen in the same component, then the arc between these points is homotopic to the line segment in the subspace defined by these points, and the line segment can be contracted to a point. Consequently, each loop is homotopic to a loop that does not have any consequent intersections in the same connected component of $\{x=0\}$. Such a loop has the same number of intersections with both components, and this number expresses, how many times the loop is wired around the cut locus, see the visualization of the cut locus in Figure \ref{cut_locus_visualization}. It means that $\pi_1(N_2)=\mathbb{Z}$.

Now, let us consider a generator of the fundamental group  $N_1$, i.e., a loop that goes once around the axis of the cylinder. Concretely, let us consider the loop in $N_1$ defined by $\tau=\pi/2$ and $\rho=\sigma=1$, i.e. the loop $$[0,1]\to N_1: t\to (\pi/2,2\pi t,1,1).$$
Its image  under $\widehat{\exp}$ is obtained by substituting the corresponding parameters into formulas for invariants \eqref{invariants} and \eqref{invariants2}
  \begin{align}
  \begin{split}
     x &=2 \cos(2\pi t), \\
   |\ell| &=\sqrt{ 4\sin^{2}(2\pi t)+ \pi^{2}}
   , \\
      |y|&= \sqrt{4 \sin^{2}(2\pi t) +\pi^{2}/4},\\
  y\cdot \ell &=-3\pi\sin(2\pi t).  \\
  \end{split}
   \end{align}
This loop in $N_2$ meets the subspace $\{x=0\}$ twice. For $t=1/4$ the intersection lies in the connected component $\varphi>\pi/2$ since $y\cdot \ell=-3\pi<0$. On the other hand, for 
$t=3/4,$ the loop goes through the other component since we get $y\cdot \ell=3\pi>0$ and thus $\varphi<\pi/2$. We conclude that the loop goes once around the cut locus, and thus, it is a generator of the fundamental group of $N_2.$ Consequently, the induced map on fundamental groups $[\widehat{\exp}]$ is bijective, and together with the previous steps, we conclude that $\widehat{\exp}:N_1\to N_2$ is a global diffeomorphism.

\end{document}